\documentclass[11pt]{amsart}

\usepackage{amsmath,amsthm, amsbsy,amssymb}
\usepackage{fullpage}
\usepackage{commath}
\usepackage{mathtools}
\usepackage{comment}
\usepackage{mathrsfs}
\usepackage{hyperref}
\usepackage{xcolor}

\theoremstyle{plain}
\newtheorem{theorem}{Theorem}
\newtheorem{lemma}[theorem]{Lemma}
\newtheorem{proposition}[theorem]{Proposition}
\newtheorem{corollary}[theorem]{Corollary}
\newtheorem{definition}[theorem]{Definition}

\numberwithin{theorem}{section}
\numberwithin{equation}{section}

\newcommand{\C}{\mathbb{C}}
\newcommand{\N}{\mathbb{N}}

\newcommand{\R}{\mathbb{R}}
\newcommand{\T}{\mathbb{T}}

\newcommand{\fG}{\mathfrak{G}}

\newcommand{\cH}{\mathcal{H}}
\newcommand{\cT}{\mathcal{T}}

\newcommand{\cO}{\mathcal{O}}

\newcommand{\cF}{\mathcal{F}}

\newcommand{\rS}{\mathrm{S}}
\newcommand{\rU}{\mathrm{U}}

\newcommand{\ip}[2]{\langle #1,#2 \rangle}

\newcommand{\intL}[3]{\int_{\C^{#1}} #2 d\lambda_{#1}(#3)}

\newcommand{\F}{\mathcal{F}(\mathbb{C}^n)}
\newcommand{\Fp}[1]{\mathcal{F}(\mathbb{C}^{#1})}

\newcommand{\Linfnn}{L^{\infty}(\mathbb{C}^{n})^G}
\newcommand{\UU}{\rU_{n_1}\times \cdots \times \rU_{n_k}}
\newcommand{\pin}{\pi^{n_1}\otimes \cdots \otimes \pi^{n_k}}
\newcommand{\polynn}{\Poly{m_1}{n_1}\otimes\cdots\otimes \Poly{m_k}{n_k}}

\newcommand{\ro}{C_{b,u}(\mathbb{N}_0,\rho_1)}
\newcommand{\RO}{C_{b,u}(\mathbb{N}_0^k, \rho_k)}
\newcommand{\tensorRO}{C_{b,u}(\mathbb{N}_0, \rho_1)\otimes \cdots\otimes C_{b,u}(\mathbb{N}_0, \rho_1)}
\newcommand{\eigenfnn}{\gamma_{\vn,a}}
\newcommand{\eigenfnnvarphi}{\gamma_{\vn,\varphi}}
\newcommand{\specnn}{\mathfrak{G}_{\vn}}

\newcommand{\eigenfn}{\gamma_{\textbf{1},a}}
\newcommand{\eigenfb}[1]{\gamma_{\textbf{1},#1}}

\newcommand{\Poly}[2]{\mathrm{P}^{#1}[\C^{#2}]}

\newcommand{\id}{\mathrm{id}}
\newcommand{\rP}{\mathrm{P}}

\newcommand{\vx}{\textbf{\textit{x}}}

\newcommand{\vm}{\textbf{\textit{m}}}
\newcommand{\vn}{\textbf{\textit{n}}}
\newcommand{\vr}{\textbf{\textit{r}}}
\newcommand{\vp}{\textbf{\textit{p}}}

\begin{document}

\keywords{Toeplitz Operators, Fock Space, Reproducing Kernel Hilbert spaces, Commutative $C^*$-algebras}
\title[Toeplitz operators with quasi-radial symbols]{Toeplitz operators on the Fock space  with quasi-radial symbols}

\author{Vishwa Dewage \& Gestur \'Olafsson}
\address{Department of Mathematics, Louisiana State University\\
Baton Rouge, LA 70803, USA\\
E-Mail Dewage:vdewag1@lsu.edu
E-Mail \'Olafsson: olafsson@math.lsu.edu}
\subjclass[2000]{22D25, 30H20, 41A35, 47B35}
 
\thanks{The research of G. \'Olafsson was partially supported by Simons grant 586106.}

\maketitle
\begin{abstract}
The Fock space $\mathcal{F}(\C^n)$ is the space of holomorphic functions on $\C^n$ that are square-integrable with respect to the Gaussian measure on $\C^n$. This space plays an
important role in several subfields of analysis and representation theory. In particular, it has for a long time been a model to study Toeplitz operators.
Esmeral and Maximenko showed in 2016 that   radial Toeplitz operators on $\mathcal{F}(\C)$ generate a 
commutative $C^*$-algebra 
which is isometrically isomorphic to  the $C^*$-algebra $\ro$. In this article, we extend the result to  
$k$-quasi-radial symbols acting on the Fock space $\F$. We calculate the spectra of 
the said Toeplitz operators and show that the set of all eigenvalue functions is dense in the 
$C^*$-algebra $\RO$ of bounded functions on $\N_0^k$ which are uniformly continuous 
with respect to the square-root metric. In fact, the $C^*$-algebra generated by 
Toeplitz operators with quasi-radial symbols is $\RO$. 
\end{abstract}

\tableofcontents

\section{Introduction}
Reproducing kernel Hilbert spaces $\cH$ of square integrable holomorphic functions on complex domains plays an
important role in several subfields of analysis and representation theory \cite{B61}. Those spaces lead to a very simple
form of quantization, \cite{B75a, B75b, EU10, EU11} turning functions into bounded operators on $\cH$, by using Toeplitz operators
$T_\varphi (f) =P(\varphi f)$, where $P$ is the orthogonal projection onto $\cH$. These operators have been studied for a long time, \cite{BC86,BI12,BL11,BS06,C94}.

Since the fundamental work \cite{GV02,GQV06}, and then later several authors have studied $C^*$-algebras generated by
subclasses of bounded symbols asking the question which of them lead to \textit{commutative} $C^*$-algebras as in
\cite{DOQ15,DOQ18,DOQ21,DQ18,EM16,GMV13,GKV03,QS11,QV07a,QV07b,QV08,V08}. In
the case of commutative $C^*$-algebras, the next natural question is to determine the
spectrum and, if possible, determine the Gelfand transform. This is often achieved by constructing a transform that maps Toeplitz operators to multiplier operators on an $L^2$ space.

In the special case of the unit ball in $\C^n$, it
was shown that 
Toeplitz operators with radial symbols acting on the Bergman space generate a commutative $C^*$-algebra as in
 \cite{GKV03} and the corresponding $C^*$-algebra  is isometrically isomorphic to the space of bounded sequences that are uniformly continuous with respect to the logarithmic 
 metric \cite{GMV13}. Toeplitz operators with radial
  symbols on the Fock space $\mathcal{F}(\C)$ generate a 
  commutative $C^*$-algebra \cite{GV02} which is isometrically isomorphic to the $C^*$-algebra $\ro$ of bounded sequences that are uniformly continuous with respect to the square-root metric \cite{EM16}. 
  Furthermore, a more general theory for weighted Bergman spaces on any bounded symmetric domain was formulated in \cite{DOQ15}, using representation theory.

In this paper, we consider the $C^*$-algebra generated by Toeplitz operators with $k$-quasi-radial symbols, discussed in \cite{BV12,MSR16,V10a,V10b,QS14,QS15}, acting on the Fock space $\F$. As an instance of the general theory discussed in \cite{DOQ15}, the generated $C^*$-algebra is commutative. 
We compute the eigenvalue functions of Toeplitz operators with quasi-radial symbols and then show that they are dense in the space $\RO$ of bounded functions on $\N_0^k$ that are uniformly continuous with respect to the square-root metric. As a consequence, the $C^*$-algebra generated by Toeplitz operators with $k$-quasi-radial symbols is isometrically isomorphic to the $C^*$-algebra $\RO$.

We now introduce some notations and recall the following well known facts about the Fock space.
Denote by $d\lambda_n$ the Gaussian measure  $d\lambda_n(z)=\frac{1}{\pi^n}e^{-|z|^2}dz$ for $z \in\C^n$,
where $dz$ is the Lebesgue measure on $\C^n\simeq \R^{2n}$. The {\it Fock space} $\F$ is
the Hilbert space of all holomorphic functions on $\C^n$ that
are square integrable with respect to $d\lambda_n$,
$$\Fp{n}=L^2(\C^n,\lambda_n)\cap \cO (\C^n).$$ The point-evalutation maps $f\mapsto f(z)$
are continuous and hence $\F$ is a reproducing kernel Hilbert space with the reproducing kernel $K$ given by
$K(w,z)=K_z(w)=e^{w\Bar{z}}$, for $(z,w)\in\C^{2n}$ (here $\Bar{z}w$ denotes the scalar product $w_1\Bar{z}_1+\cdots +w_k\Bar{z}_k$). Thus for $z\in \C^n$ and $f\in \F$ we have 
$$f(z)=\langle f,K_z \rangle$$
with the inner product of $L^2(\C^n,\lambda_n)$.

The Bergman projection $P: L^2(\C^n)\rightarrow \F$ is given by 
$$Pf(z)=\ip{f}{K_z}.$$ $P$ is a bounded linear operator with $\|P\|=1$.

Given $\varphi\in L^\infty(\C^n)$, we define the Toeplitz operator $T_\varphi:\F \rightarrow \F$ by
    $$T_\varphi f(z)= P(\varphi f)(z)=  \int_{\C^n} \varphi(w)f(w)\overline{K_z(w)}d\lambda_n(w).$$
As the multiplier operator $M_\varphi : L^2(\C^n,\lambda )\to L^2(\C^n,\lambda )$ is bounded
of norm $\|\varphi \|_\infty$ and $P$ is bounded, it follows that $T_\varphi$ is bounded and
$\|T_\varphi \|\le \|\varphi \|_\infty$. The function $\varphi$ is called the symbol of the Toeplitz operator $T_\varphi$.

 Grudsky and Vasilevski \cite{GV02} showed that the eigenvalue sequences of radial Toeplitz operators on $\mathcal{F}(\C)$ are of the form
 $$\gamma_{1,a}(m)=\frac{1}{m!}\int_0^{\infty} a(\sqrt{r})r^{m}e^{-r}dr.$$

  We prove that the Toeplitz operators with $k$-quasi-radial symbols diagonalize with eigenvalue functions given by 
  $$\eigenfnn(\vm)=\frac{1}{(\vm+\vn-\textbf{1})!}\int_{\R_+^k} a(\sqrt{r_1},\dots,\sqrt{r_k})r^{\vm+\vn-\textbf{1}}e^{-(r_1+\cdots+r_k)}dr$$
  for any $\vm\in\N_0^k$, where $\vn=(n_1,\dots,n_k)$.
  
  As a corollary the eigenvalue sequences of radial Toeplitz operators on $\F$ are of the form
 $$\gamma_{n,a}(m)=\frac{1}{(m+n-1)!}\int_0^{\infty} a(\sqrt{r})r^{m+n-1}e^{-r}dr$$
  which are the $(n-1)^\text{th}$ left shift of $\gamma_{1,a}$.

To prove eigenvalue function of Toeplitz operators with $k$-quasi-radial symbols is dense in the $C^*$-algebra $\RO$, we follow some of the ideas in the work by Esmeral and Maximenko \cite{EM16}.\\

\noindent \textbf{Acknowledgements:} We sincerely thank the anonymous reviewers for the constructive and positive feedback.
 

\section{Toeplitz operators with bounded  \texorpdfstring{$k$}{k}-quasi-radial symbols}\label{sec:diagnn}

In this section, we introduce quasi-radial symbols and diagonalize the Toeplitz operators with quasi-radial symbols. 

We begin with some notations. Let $k\in \N$ and let $\vn=\vn(k)=(n_1,\dots,n_k)\in \N^k$ and let $n=n_1+\cdots +n_k$. Let $\C^{\vn}=\C^{n_1}\times\cdots\times\C^{n_k}$.
We identify $z\in \C^{n}$ by $z=(z_{(1)},\cdots,z_{(k)})\in \C^{\vn}$ where $z_{(j)}\in\C^{n_j}$. Consider the group $G=\UU$ where $\rU_{n_i}$ are $n_i\times n_i$ unitary matrices. They act on  $f$ defined on $\C^n$ by
$(A_1,\dots A_k)f(z)=(A_1,\dots, A_k)f(z_{(1)},\cdots,z_{(k)}):=f(A_1^{-1}z_{(1)},\dots A_{k}^{-1}z_{(k)})$ which is the usual action of matrices we get by considering $G$ as a subgroup of $\rU_n$. The action leaves the Gaussian measure $d\lambda_n(z)=d\lambda_{n_1}(z_{(1)})\dots d\lambda_{n_k}(z_{(k)})$ invariant.

\begin{definition} Let $\varphi :\C^n\to \C$ be measurable. The function
$\varphi$ is said to be
$G$-invariant if for all $(A_1,\dots A_k)\in G$ and $z\in\C^n$ we have
$(A_1,\dots A_k)\varphi(z)=\varphi (A_1^{-1}z_{(1)},\dots A_{k}^{-1}z_{(k)}) =\varphi(z)$.
\end{definition}

 Define the set of all $k$-quasi-radial symbols, denoted by $\Linfnn$, to be the set of all essentially bounded functions on $\C^n$ that are $G$-invariant. The case $k=1$ corresponds to radial symbols and the case $n=k$ corresponds to separately radial symbols. Now we take into consideration the $C^*$-algebra $\cT_G$ generated by Toeplitz operators with symbols in $\Linfnn$. 
The main objective of this section is to diagonalize the Toeplitz operators with symbols in $\Linfnn$.

\subsection{A classification of the class of symbols \texorpdfstring{$\Linfnn$}{}} 

\begin{definition} Let $\varphi :\C^q\to \C$ be measurable. The function
$\varphi$ is said to be,
\begin{enumerate}
\item radial if there exists a measurable function $a_\varphi:\R_+ 
\to\C$ such that $\varphi (z)=a_\varphi(|z|)$ for $z\not= 0$.
\item $\rU_q$-invariant if for all $A\in\rU_q$ and $z\in\C^q$ we have
$\varphi (Az) =\varphi(z)$.
\end{enumerate}
\end{definition}
Following lemma is a well-known fact.

\begin{lemma}  \label{lem:radialf}
The measurable function $\varphi :\C^{q}\to \C$ is $\rU_q$-invariant if and only if there exists $a:\R_+\rightarrow \C$ such that
$\varphi(z)=a(|z|)$.
\end{lemma}

\begin{proof} As $|Az|=|z|$ for $A\in\rU_q$, it follows that any function of the form
$\varphi(z)=a(|z|)$ is $\rU_q$-invariant. As $z\mapsto |z|$
is continuous it follows that $\varphi$ 
is measurable if and only if $a$ is measurable.

Assume that $\varphi$ is $\rU_q$-invariant. Define $a_\varphi(r)=\varphi (re_1)$. 
Then $a_\varphi$ is measurable. Let $z\in\C^q$. Then, as $\rU_q$ acts
transitively on the sphere $\rS^{2q-1}$,
there exist $A\in\rU_q$ such that
$Az=|z|e_1$. Thus
\[\varphi(z)=\varphi (Az) =\varphi (|z|e_1) =a_\varphi (|z|).\]
Thus $\varphi (z)=a_\varphi (|z|)$.
\end{proof}

The following lemma is a consequence of Lemma \ref{lem:radialf}.

\begin{lemma} \label{lem:radialfn}
The measurable function $\varphi :\C^{\vn}\to \C$ is $G$-invariant if and only if there exists $a:\R_+^k\rightarrow \C$ such that
$\varphi(z)=a(|z_{(1)}|,\dots ,|z_{(k)}|)$.
\end{lemma}
  
\subsection{Diagonalization by Schur's lemma}Let $\pi$ and $\sigma$ be unitary representations of a topological group H. A continuous linear map 
$T:V_\pi\to V_\sigma$ is an \textit{intertwining operator} if for all $A\in H$ commutating
relation $T\pi (A)=\sigma (A)T$ holds. $\pi $ and $\sigma$ are equivalent if there exists
an unitary isomorphism that intertwines $\pi$ and $\sigma$. Finally if $\pi $ is irreducible 
and $T$ is an intertwining operator then Schur's lemma says that $T=\lambda \id$ for
some $\lambda\in \C$.

Since the action of $G$ is defined by the action of $\rU_{n_i}$ on $\cF^2(\C^{n_i}),i=1,\dots,k$, first we consider the action of $U_q$ on $\cF^2(\C^q), q\in\N$.

Denote by $\Poly{m}{q}$ the space of
homogeneous holomorphic polynomials on $\C^q$ of degree $m$. Note that the space of holomorphic polynomials is dense in $\cF (\C^q)$ and any holomorphic polynomial can be written in a unique way as a direct sum of homogeneous polynomials. In fact the following is well known and can be found in \cite{Z12}.

\begin{lemma} For $\alpha \in\N^q$ let $p_\alpha (z)=z_1^{\alpha_1}\cdots z_q^{\alpha_q}$. Then
\[\|p_\alpha\|^2_\cF = \alpha !=\alpha_1!\ldots \alpha_q!\]
and the collection
\[\{q_\alpha =\frac{1}{\sqrt{\alpha !}} p_\alpha\}_{\alpha\in\N^q}\]
forms an orthonormal basis for $\cF(\C^q)$.
\end{lemma}
  
$\rU_q$ acts on functions on $\C^q$ by $Af(z)=f(A^{-1}z)$. This action leaves
$\cF(\C^q)$ invariant and, as the measure $\lambda_q$ is $\rU_q$-invariant, it is
a unitary representation denoted by $\pi^q$.  This representation is
reducible. In fact this representation leaves the spaces $\Poly{m}{q}$ invariant. We denote
the corresponding representation by $\pi_m^q$. The representation $\pi_m^q$ is irreducible.
The torus $\T\id $ is the center of $\rU_q$ and acts by the character $z\mapsto z^{-m}\id$. The
representations  
$\pi_m^q$ and $\pi_k^q$ are inequivalent if $m\not= k$ as can easily be seen as the
action of the center is different. The same conclusion follows from
the fact that $d_m=\dim_\C \Poly{m}{q}\not= d_k$ if $m\not= k$ for $q>1$ and $d_m=1$ for $q=1$. The decomposition
of $\Fp{q}$ in irreducible $U_q$ representations is given by
$\Fp{q}=\bigoplus_{m=0}^\infty \Poly{m}{q}$.

\begin{definition}
Let  $\pi^{n_i}$ be the unitary representations of $\rU_{n_i}$, acting on $\Fp{n_i}$, $i=1,\dots,k$, as in the above discussion. The outer tensor product of the representations $\pi_{p_i}$, denoted $\pin$ is defined by
$$\pin(A_1,\dots A_k)=\pi^{n_1}(A_1)\otimes\cdots\otimes \pi^{n_k}(A_k)$$
for all $(A_1,\dots A_k)\in \UU=G$. Here $\pi^{n_1}(A_1)\otimes\cdots\otimes \pi^{n_k}(A_k)$ is the tensor product of the operators $ \pi^{n_i}(A_i)$ acting on the Hilbert spaces $\Fp{n_i}$.
\end{definition}

 The outer tensor product $\pin$ is a unitary representation of $G$ acting on  $\Fp{n_1}\otimes\cdots\otimes \Fp{n_k}=\F$. For more details see \cite{FO14}.
 
 \begin{lemma}\label{lem:intertwines}
Let $\varphi \in \Linfnn$. Then $T_\varphi$ intertwines with $\pin$.
 \end{lemma}
    
 \begin{proof}
Let $A_i\in \rU_{n_i}$ and $f_i\in \Fp{n_i}$, $i=1,\dots,m$. Notice that
\begin{align*}
    T&_{\varphi}  (\pin) (A_1,\dots A_k)  f_{1}\otimes\cdots\otimes f_{k}(z)= T_\varphi f(A_1^{-1}z_{(1)},\dots ,A_k^{-1}z_{(k)})\\
    &=\intL{n}{\varphi(w)f(A_1^{-1}z_{(1)})\cdots f(A_k^{-1}z_{(k)})\overline{k_z(w)}}{w}\\
    &=\intL{n}{\varphi(A_1w_{(1)},\dots A_{k}w_{(k)})f_{1}\otimes\cdots\otimes f_{k}(w)\overline{k_z(A_1w_{(1)},\dots A_{k}w_{(k)})}}{w}\\
    &(\text{by the change of variable $(w_{(1)},\dots w_{(k)})\rightarrow(A_1w_{(1)},\dots, A_kw_{(k)})$}).   
\end{align*}
Since 
\begin{align*}
    k_z(A_1w_{(1)},\dots A_{k}w_{(k)})& = e^{A_1w_{(1)}\overline{z_{(1)}}+\cdots +A_kw_{(k)}\overline{z_{(k)}}}\\
    &=e^{w_{(1)}\overline{(A^*z_{(1)})}+\cdots +w_{(k)}\overline{A_{k}^*z_{(k)}}}\\
    &=e^{w_{(1)}\overline{(A^{-1}z_{(1)})}+\cdots +w_{(k)}\overline{A_{k}^{-1}z_{(k)}}} \hspace{.5cm} \text{(as $A_i$ are unitary)}\\
    &=k_{(A_1^{-1}z_{(1)},...A_{k}^{-1}z_{(k)})}(w)
\end{align*}
and $\varphi$ is $G$-invariant, 
\begin{align*}
    T_\varphi  (\pi^{n_1}\times \cdots \times & \pi^{n_k})  (A_1,\dots A_k)  f_{1}\otimes\cdots\otimes f_{k}(z) \\
   &= \intL{n}{\varphi(w)f_{1}\otimes\cdots\otimes f_{k}(w)\overline{ k_{(A_1^{-1}z_{(1)},...A_{k}^{-1}z_{(k)})}(w)}}{w}\\
   &=(\pin)(A_1,\dots A_k)T_\varphi f(z). \qedhere
\end{align*}

 \end{proof}

 Since the tensor product of irreducible representations is irreducible (see Proposition 6.75 in \cite{FO14} for a proof), $\Fp{n}$ can be decomposed by irreducible sub-representations of $\pin$ as 
$$\Fp{n}=\bigoplus_{m_1,\dots,m_k=0}^\infty \polynn.$$ 
Hence if  $\varphi \in \Linfnn$, $T_\varphi|_{\polynn}=\eigenfnnvarphi(\vm)\id$ for some $\eigenfnnvarphi(\vm)\in \C$ for all $\vm=(m_1,\dots,m_k)\in \N_0^k$ by Schur's lemma.\\

As a consequence of Lemma \ref{lem:radialfn}, we identify the class of symbols $\Linfnn$ with  essentially bounded functions $a:\R_+^k\rightarrow \C$ and denote the corresponding Toeplitz operator and its spectrum by $T_{\vn,a}$ and $\eigenfnn$.

Since the above decomposition of $\Fp{n}$ is multiplicity free and $G$ is compact, the set of all operators that intertwine with $\pin$, denoted $\text{End}_G(\Fp{n})$, is commutative (see in particular Proposition 4.1 and Theorem 6.4 in \cite{DOQ15} for more details). Then as a consequence of Lemma \ref{lem:intertwines} we have the following corollary.

\begin{corollary}
The C$^*$-algebra generated by Toeplitz operators with $G$-invariant symbols is commutative.
\end{corollary}

\subsection{Computing eigenvalue functions}

We denote by $\rS^{q} =\{x\in\R^{q+1}\mid |x|=1\}$ the $q$-dimensional sphere
in $\R^{q+1}$. We denote by $\sigma$ the
unique $\rU_{q+1}$-invariant  
measure on $\rS^{q}$ that satisfy the polar coordinates formula given by
$$\int_{\R^{q+1}}f(x) dx= \int_0^\infty \int_{\rS ^{q}}f(r\omega)d\sigma(\omega)r^{q}dr$$
for any $f\in L^1(\R^{q+1},dx)$ where $dx$ is the Lebesgue measure on $\R^{q+1}$. See section 2.7 in \cite{F99} for more details.

\begin{lemma} Let $f\in L^1(\C^q,dz)$. Then
\[\int_{\C^q}f(z)dz=\int_0^\infty \int_{S^{2q-1}} f(r\omega) d\sigma (\omega) r^{2q-1}dr .\]
In particular if $f(z)=a(r)$ is $\rU(q)$-invariant then
\[\int_{\C^q} f(z)dz=\sigma(\rS^{2q-1}) \int_0^\infty a(r)r^{2q-1}dr .\]
\end{lemma}


We recall the following well known fact:
\begin{lemma}\label{lem:Ip} Let $f,g\in\cF^2(\C^q)$ and $\varphi \in L^\infty (\C^q)$. Then
\[ \ip{T_\varphi f}{g}= \ip{\varphi f}{g}=\frac{1}{\pi^q}\int_{\C^q} \varphi(z)f(z)\overline{g(z)}e^{-|z|^2}dz.\]
In particular for $f=g=p_\alpha$ we get
\[\ip{T_\varphi p_\alpha}{p_\alpha} =\frac{1}{\pi^q} \int_{\C^q} \varphi (z)\prod_{j=1}^n |z_j|^{2\alpha_j}e^{-|z_j|^2}dz_1\cdots dz_n. \]
\end{lemma}
\begin{proof} Let $P: L^2(\C^q)\to \cF^2(\C^q)$ be the orthogonal projection. Then, as $g\in \cF^2(\C^q)$, we get
\[\ip{T_\varphi f}{g}=\ip{P(\varphi f)}{g} = \ip{\varphi f}{g}.\qedhere\]
\end{proof}

The following is well known and can be found for more general
situations   in \cite{FK94} and other places. We also point to
\cite{F01} for a general discussion on how to
integrate polynomials over spheres:

\begin{lemma}\label{lem:norms} Let $p_{\alpha}\in \rP^m(\C^q) $. Then
$\displaystyle \|p_{\alpha}|_{\rS^{2q-1}}\|^2 =  \frac{2\pi^q}{\Gamma (q+m)} \|p_{\alpha}\|^2_{\cF}$.
In particular  
\[\|q_\alpha|_{\rS^{2q-1}}\|^2 =\frac{2\pi^q}{\Gamma (q+m)}.\]
\end{lemma}
\begin{proof}  We have $p_{\alpha} (r\omega)=r^mp(\omega) $, $r>0$. Hence: 
\begin{align*}
\|p_{\alpha}\|^2_{\cF}&= \frac{1}{\pi^q} \int_{0}^\infty \int_{\rS^{2q-1}}|p_{\alpha}(r\omega )|^2d\sigma (\omega)
e^{-r^2}r^{2q-1}dr\\
&=   \frac{1}{2\pi^q}\left(\int_{\rS^{2q-1}} |p_{\alpha}(\omega )|^2d\sigma(\omega)\right)
\left(\int_0^\infty r^{2m+2q-2}e^{-r^2}(2rdr)\right)\\
&=\frac{1}{2\pi^q }\|p_{\alpha}|_{\rS^{2q-1}}\|^2\int_0^\infty u^{m+q-1}e^{-u}du\qquad u=r^2,\, du =2rdr\\
&= \frac{1}{2\pi^q }\|p_{\alpha}|_{\rS^{2q-1}}\|^2\Gamma (m+q).\qedhere
\end{align*}
\end{proof}

We recall that the $C^*$-algebra $\cT_G$ generated by $G$-invariant bounded symbols  is
commutative and acts by scalars on each of the spaces $\polynn$.

\begin{theorem}\label{thm:specnn}
Assume  $\varphi \in L^\infty (\C^n)^G$  let $\vm =(m_1,\dots,m_k)\in {\N_0^k}$. Then 
$$ T_\varphi|_{\polynn}=\gamma_{n,\varphi} (\vm)\id=\gamma_{n,a}(\vm )\id $$
where
\[
 \gamma_{\vn,a}(\vm) =
    \frac{1}{ (\vm + \vn-\textbf{1})! }\int_{\R_+^k} a_\varphi (\sqrt{\vr})\vr^{\vm+\vn-\textbf{1}} 
    e^{-(r_1+\cdots  +r_k)}dr_1 \dots dr_k.\] 
where we use the notations $\sqrt{\vr}=(\sqrt{r_1},\dots,\sqrt{r_k})$ and $\textbf{1}=(1,\dots, 1)\in \N^k$.
\end{theorem}
\begin{proof} Assume that $\varphi (z)= a(|z_{(1)}|,\ldots ,|z_{(k)}|)=a(r_1,\ldots ,r_k)$. Let $\alpha_j\in\N_0^{n_j}$ such that $|\alpha_j|=m_j$ and let $q_{\alpha_j}\in \Poly{m_j}{n_j}$ be the normalized monomials defined in Lemma \ref{lem:norms}. Recall from Lemma \ref{lem:norms} that
\[\|q_{\alpha_j}|_{\rS^{2n_j-1}}\|^2=\frac{2\pi^{n_j}}{\Gamma(n_j+m_j)} \]
and hence
$$\prod_{j=1}^k \|q_{\alpha_j}|_{\rS^{2n_j-1}}\|^2=\frac{2^k\pi^n}{(\vm+\vn-\textbf{1})!}.$$
Let $\boldmath{\alpha} = (\alpha_1,\ldots, \alpha_k)$ and define the normalized monomial $q_{\boldmath{\alpha}}$ by
$$q_{\boldmath{\alpha}}(z)=\prod_{j=1}^k q_{\alpha_j}(z_{(j)}).$$
Then $q_{\boldmath{\alpha}}\in\polynn$ and hence
\[\gamma_{n,\varphi}(\vm) =\langle 
T_\varphi q_\alpha, q_\alpha \rangle.
\]

As $\varphi$ is $G$-invariant and bounded we can use polar-coordinates on $\C^{n_j}$, Fubini's theorem, Lemma \ref{lem:radialfn}, Lemma \ref{lem:Ip} and the fact that $q_{\alpha_j}$ are homogeneous of order $m_j$,
to write
\begin{align*}
\ip{T_\varphi q_\alpha}{ q_\alpha }&=\frac{1}{\pi^n}\int_{\C^n}\varphi (z)|q_{\alpha}(z) |^2 e^{-|z|^2}dz\\
&= \frac{1}{\pi^n}\int_{\C^{n_1}}\!\!\cdots \!\! \int_{\C^{n_k}} 
a (|z_{(1)}|,\ldots ,|z_{(k)}|)\prod_{j=1}^k |q_{\alpha_j}(z_{(j)})|^2 e^{-|z_{(j)}|^2}dz_{(1)}\ldots dz_{(k)} \\
&=\frac{1}{\pi^n}\int_0^\infty\!\!\cdots \!\!\int_0^\infty a (\vr)\prod_{j=1}^k r_j^{2m_j+2n_j-1}e^{-r_j^2}dr_1\ldots dr_k
  \prod_{j=1}^k \|q_{\alpha_j}|_{\rS^{2n_j-1}}\|^2\\
 &= \frac{2^k}{(\vm+\vn-\textbf{1} )!}\int_0^\infty\!\!\cdots \!\!\int_0^\infty a (\vr )\prod_{j=1}^k r_j^{2m_j+2n_j-1}e^{-r_j^2}dr_1\ldots dr_k\\ 
&= \frac{1}{( \vm+\vn-\textbf{1} )! }\int_0^\infty\!\!\cdots \!\!\int_0^\infty a (\sqrt{\vr})\prod_{j=1}^k r_j^{m_j+n_j-1} e^{-r_j} 
dr_1\cdots dr_k,\end{align*}
where we in the last line used the substitutions $u=r_j^2$ and hence $du =2r_jdr_j$.
\end{proof}

\begin{corollary} \label{coro:eigenshift}
If $\vn=(n_1,\dots,n_k)$ and $\vm=(m_1,\dots m_k)$, then
$$\eigenfnn(\vm) = \eigenfn(m_1+n_1-1,\dots,m_k+n_k-1)=\eigenfn(\vm+\vn-\textbf{1}).$$
\end{corollary}

Note that eigenvalue functions $\eigenfn$ given above corresponds to Toeplitz operators with separately radial symbols. In particular we obtain the following special cases:
\begin{corollary}[Grudsky and Vasilevski \cite{GV02}]  Assume that $n=1$
\[\gamma_{1,a}(m) =\frac{1}{m! }\int_0^\infty a(\sqrt{r}) r^{m}e^{-r}dr.\] 
\end{corollary}
If $k=1$, then $\vn=(n)$ and we have the following corollary.
\begin{corollary} If $k=1$ (radial symbols) then
\[\gamma_{n,a}(m) =\frac{1}{(m+n-1)!} \int_0^\infty a(\sqrt{r})r^{m+n-1}e^{-r}dr
= \gamma_{1,a}(m+n-1).\]
\end{corollary}
If $n=k$, then $\vn=(1,\dots,1)=\textbf{1}$ and we have the following corollary. 
\begin{corollary}
If $n=k$ (separately radial symbols) then
$$\eigenfn(\vm) =
    \frac{1}{\vm ! }\int_{\R_+^k}  a_\varphi (\sqrt{\vr})\vr^{\vm} 
    e^{-(r_1+\cdots  +r_k)}dr_1 \dots dr_k.$$
\end{corollary}
\begin{corollary} \label{coro:eigenproduct}If $a(r_1,\ldots ,r_k)= \prod_{j=1}^k a_j(r_j)$ then
\[\gamma_{n,a}(m_1,\ldots ,m_k) =  \prod_{j=1}^k \gamma_{1,a_j}(m_j+n_j-1).\]
\end{corollary}

\section{The space \texorpdfstring{$\RO$}{Cbu}}

In this section we introduce C$^*$-algebra $\RO$, the set of all bounded functions on $\N_0^k\times \N_0^k$ that are uniformly continuous with respect to the square-root metric. These are sometimes called \textit{Square root-slowly oscillating} functions.

For $k\in \N$, let $\rho_k:\N_0^k\times \N_0^k \rightarrow [0,\infty)$ be given by
$$\rho(\vm,\vm')= |\sqrt{m_1}-\sqrt{m'_1}|+\cdots +|\sqrt{m_k}-\sqrt{m'_k}|.$$
for all $\vm=(m_1,\dots,m_k),\vm'=(m'_1,\dots,m'_k)\in \N_0^k$
Then $\rho_k$ is a metric on $\N_0^k$. \textit{The modulus of continuity} with respect to the metric $\rho_k$ of a function $\sigma \in l^\infty(\N_0^k)$ is the function $\omega_{\rho_k, \sigma}:[0,\infty)\rightarrow [0,\infty)$ given by
$$\omega_{\rho_k, \sigma}(\delta)= \sup\{|\sigma(\vm)-\sigma(\vm')|:\rho(\vm,\vm')\leq \delta\}.$$

$\RO$ is the set of all bounded functions on $\N_0^k$ that are uniformly continuous with respect to the square-root metric $\rho_k$:
$$\big\{\sigma\in l^\infty(\N_0^k) : \lim_{\delta\rightarrow 0}\omega_{\rho_k, \sigma}(\delta)=0  \big\}.$$ 
$\RO$ is a closed subalgebra of $l^\infty(\N_0^k)$.\\

\begin{definition}\label{def:shifts}
Let $\sigma \in\RO$ and let $\textbf{s}=(s_1\dots,s_k)\in\N_0^k$. We define the left and right shift operators on $\RO$ with respect to $\textbf{s}$, denoted $\tau_L^{\textbf{s}}$ and $\tau_R^{\textbf{s}}$ respectively,  by 
$$\tau_L^{\textbf{s}}\sigma(\vm)=\sigma(\vm+\textbf{s}) \quad \text{and} \quad
\tau_R^{\textbf{s}}\sigma(\vm)=\left\{
\begin{array}{ll}
      \sigma(\vm-\textbf{s}) & ;m_i\geq s_i\ \forall \ i=1,\dots,k\\
      0 & ;otherwise\\
\end{array} 
\right.$$
for all $\vm=(m_1,\dots,m_k)\in \N_0^k$.
\end{definition}

Now we present the following two lemmas that will be used in section \ref{sec:densitynn}.

\begin{lemma}\label{lem:lshiftop}
Let $\sigma \in \RO$ and $\textbf{s}=(s_1\dots,s_k)\in\N_0^k$. Then $ \tau_L^{\textbf{s}} \sigma \in \RO.$
\end{lemma}
\begin{proof}
Boundedness of $\tau_L^{\textbf{s}}\sigma$ follows easily from the boundedness of $\sigma$. Let $\vm=(m_1,\dots,m_k),\\ \vm'=(m'_1,\dots,m'_k)\in \N_o^k$.
By observing the square-root function, we have, $$|\sqrt{m_i+s_i}-\sqrt{m'_i+s_i}|\leq|\sqrt{m_i}-\sqrt{m'_i}|, \text{ for } i=1,\dots k.$$ Hence
$\omega_{\rho_k,\tau_L^{\textbf{s}}\sigma}(\delta)\leq \omega_{\rho_k,\sigma}(\delta)$.
\end{proof}

\begin{lemma}\label{lem:rshiftop}
Let $\sigma \in \RO$ and $\textbf{s}=(s_1\dots,s_k)\in\N_0^k$. Then $ \tau_R^{\textbf{s}}\sigma \in \RO.$

\end{lemma}
\begin{proof}
Again boundedness of $\tau_R^{\textbf{s}}\sigma$ follows from the boundedness of $\sigma$. 
Let $s=\max\{s_1,\dots,s_k\}$. Assume $\delta<\sqrt{2s}-\sqrt{2s-1}$ and let $\vm=(m_1,\dots,m_k), \vm'=(m'_1,\dots,m'_k)\in \N_o^k$ such that $\rho_k(\vm,\vm')<\delta$. Then $|\sqrt{m_i}-\sqrt{m'_i}|<\sqrt{2s}-\sqrt{2s-1}$ and since 
$$\min_{\substack{j,l\in \N_0 \\ i,j\leq 2s}}|\sqrt{j}-\sqrt{l}|= \sqrt{2s}-\sqrt{2s-1},$$ 
we have $m_i,m'_i>2s\geq 2s_i$ for all $i=1,\dots,k$. 
Notice that
\begin{align*}
    m_i,m'_i>2s_i=\frac{s_i}{1-\frac{1}{2}} &\implies \frac{1}{2}m_i< m_i-s_i,\ \frac{1}{2}m'_i< m'_i-s_i \\
    &\implies \frac{1}{\sqrt{2}}\sqrt{m_i}<\sqrt{m_i-s_i}, \ \frac{1}{\sqrt{2}}\sqrt{m'_i}<\sqrt{m'_i-s_i}\\
     &\implies \frac{1}{\sqrt{2}}(\sqrt{m_i}+\sqrt{m'_i})< \sqrt{m_i-s_i}+\sqrt{m'_i-s_i}\\
     &\implies \frac{|m_i-m'_i|}{\sqrt{m_i-s_i}+\sqrt{m'_i-s_i}}<\sqrt{2}\frac{|m_i-m'_i|}{\sqrt{m_i}+\sqrt{m'_i}}\\
     &\implies |\sqrt{m_i-s_i}-\sqrt{m'_i-s_i}|< \sqrt{2}|\sqrt{m_i}-\sqrt{m'_i}|.
\end{align*}
Hence
$\omega_{\rho_k,\tau_L^{\textbf{s}}\sigma}(\delta)\leq \omega_{\rho_k,\sigma}(\sqrt{2}\delta)$ for all $\delta<\sqrt{2s}-\sqrt{2s-1}$.
\end{proof}

The metric $\rho_k$ can be extended to $\R_+^k$ and we denote the set of all functions on $\R_+^k$ that are uniformly continuous with respect to $\rho_k$ by $C_{b,u}(\R_+^k,\rho_k)$. The following lemma shows that any $\sigma \in \RO$ can be extended to some $f\in C_{b,u}(\R_+^k,\rho_k)$. We will use this fact in section \ref{sec:density}.

\begin{lemma}\label{lem:extend}
Let $\sigma \in \RO$. Define $f$ on $\R_+^k$ by 
\begin{align*}
    f(x)&=f_\vm(x):=\sigma(\vm)+
    \sum_{l=1}^k\sum_{\substack{i_1,\dots,i_l\in\{1,\dots,k\} \\ i_1<\cdots<i_l}}a^l_{i_1,\dots,i_l}(\vm)\prod_{q=1}^l\frac{\sqrt{x_{i_q}}-\sqrt{m_{i_q}}}{\sqrt{m_{i_q}+1}-\sqrt{m_{i_q}}} 
\end{align*}
where $m_i=\lfloor x_i\rfloor$, $i=1,\dots ,k$ and the coefficients $a^l_{i_1,\dots,i_l}(\vm)$ are given by 
$$a^l_{i_1,\dots,i_l}(\vm)= \sum_{q=0}^{l} (-1)^{l-q} \sum_{\substack{j_1,\dots,j_{q}\in \{i_1,\dots,i_{l}\} \\ j_1<\cdots<j_{q}}} \sigma(\vm +e_{j_1}+\cdots + e_{j_q})$$
where $e_i$ denotes the standard basis of $\N_0^k$.

Then $f\in \RO$ and $f|_{\N_0^k}=\sigma$. Moreover $\|f\|_{\infty}=\|\sigma\|_\infty$.
\end{lemma}

The function $f$ defined above is motivated by the interpolation formula for the case $k=1$, given by
$$f(x)=\sigma(m)+(\sigma(m+1)-\sigma(m))\frac{\sqrt{x}-\sqrt{m}}{\sqrt{m+1}-\sqrt{m}}$$
where $m=\lfloor x\rfloor$ and $x\in \R_+$. The proof of Lemma \ref{lem:extend} is a generalization of the proof of Lemma 3.3 in \cite{EM16}. However, the proof depends heavily on combinatorics and induction; we have provided the proof in the appendix.


\section{Uniform continuity of eigenvalue functions with respect to the square-root metric}\label{sec:inclusion}

In this section, we show that the set of all eigenvalue functions is a subset of $\RO$.

\begin{definition}\label{def:spec}
Let $\specnn$ be the set of all eigenvalue functions: $$\specnn:=\{\eigenfnn \mid \varphi \in \Linfnn\}.$$
\end{definition}

\begin{proposition} \label{prop:inclusion}
The eigenvalue functions $\eigenfnn$ are bounded functions that are uniformly continuous with respect to the square-root metric $\rho_k$, i.e.,
$$\specnn \subseteq \RO.$$
\end{proposition}


The boundedness of $\eigenfnn$ follows from the boundedness of the symbol. We proceed to prove the uniform continuity of $\eigenfnn$ with respect to the metric $\rho_k$. The following lemma is Lemma 4.3 in \cite{EM16} which is used by Esmeral and Maximenko to prove the above theorem for $n=1$. Then we prove Lemma \ref{lem:Kbound} and Lemma \ref{lem:intproduct} which are used to prove Proposition \ref{prop:Lipchitz}. Theorem \ref{prop:inclusion} follows as a corollary of Proposition \ref{prop:Lipchitz}.

\begin{lemma}\label{lembound} For any $m\in\N$,
$$\int_{0}^{\infty}\Big|\frac{r^m}{m!}-\frac{r^{m-1}}{(m-1)!}\Big|e^{-r}dr \leq \sqrt{\frac{2}{\pi m}}.$$
\end{lemma}

As a consequence of the above lemma we prove the following:

\begin{lemma}\label{lem:Kbound} For any $m,m' \in\N_0$,
$$\int_{0}^{\infty}\Big|\frac{r^{m}}{m!}-\frac{r^{m'}}{m'!}\Big|e^{-r}dr \leq 2\sqrt{\frac{2}{\pi}}|\sqrt{m}-\sqrt{m'}|.$$
\end{lemma}

\begin{proof}
Let $\underline{m}=\min \{m,m'\}$ and $\overline{m}=\max\{m,m'\}$. Notice that
\begin{align*}
    \int_{0}^{\infty}\Big|\frac{r^{m}}{m!}-\frac{r^{m'}}{m'!}\Big|e^{-r}dr &\leq
    \sum_{i=\underline{m}+1}^{\overline{m}} \int_{0}^{\infty}\Big|\frac{r^i}{i!}-\frac{r^{i-1}}{(i-1)!}\Big|e^{-r}dr\\
    & \leq \sum_{i=\underline{m}+1}^{\overline{m}} \sqrt{\frac{2}{\pi i}} \quad \text{  (by Lemma \ref{lembound})}\\
    &= \sqrt{\frac{2}{\pi}} \sum_{i=\underline{m}+1}^{\overline{m}} \frac{(\sqrt{i}+\sqrt{i-1})(\sqrt{i}-\sqrt{i-1})}{\sqrt{i}}\\
    &\leq \sqrt{\frac{2}{\pi}} \sum_{i=\underline{m}+1}^{\overline{m}} \frac{2\sqrt{i}(\sqrt{i}-\sqrt{i-1})}{\sqrt{i}}= 2\sqrt{\frac{2}{\pi}}|\sqrt{m}-\sqrt{m'}|.\qedhere
\end{align*}
\end{proof}

\begin{lemma}\label{lem:intproduct}
Suppose $f_i,g_i:\N_0\times\R_+\rightarrow \C$ satisfy $\int_{\R_+} |f_i(m,r)| dr,\ \int_{\R_+} |g_i(m,r)| dr\leq 1$ for $i=1,\dots,k$. Then for all $m_i,m'_i\in\N_0$,
\begin{align*}
    \int_{\R_+^k}   \Big|\prod_{i=1}^k f_i(m_i,r_i) -\prod_{i=1}^k g_i(m'_i,r_i)\Big| dr_1 \cdots dr_k
    \leq \sum_{i=1}^k \int_{\R_+} |f_i(m_i,r_i)-g_i(m'_i,r_i)| dr_i.
\end{align*}
\end{lemma}

\begin{proof}
Notice that the statement is true for $k=1$. Assume the statement is true for $k$. Then
\begin{align*}
     \int_{\R_+^{k+1}}  \Big| & \prod_{i=1}^{k+1} f_i(m_i,r_i) -\prod_{i=1}^{k+1} g_i(m'_i,r_i)\Big| dr_1 \cdots dr_{k+1} \\
     & \leq \int_{\R_+}|f_{k+1}(m_{k+1},r_{k+1})-g_{k+1}(m_{k+1},r_{k+1})|dr_{k+1} \Bigg |\int_{\R_+^k}\prod_{i=1}^{k} f_i(m_i,r_i) dr_1\dots dr_k\Bigg |\\
     &\hspace{1cm} + \int_{\R_+} |g_{k+1}(m_{k+1},r_{k+1})|dr_{k+1}\int_{\R_+^k}\Big|\prod_{i=1}^{k} f_i(m_i,r_i) -\prod_{i=1}^{k} g_i(m'_i,r_i)\Big| dr_1 \cdots dr_{k}\\
     &\leq \sum_{i=1}^{k+1} \int_{\R_+} |f_i(m_i,r_i)-g_i(m'_i,r_i)| dr_i
\end{align*}
by the assumption and because
\begin{align*}
    \Bigg |\int_{\R_+^k}\prod_{i=1}^{k} f_i(m_i,r_i) dr_1\dots dr_k\Bigg | 
    &\leq \prod_{i=1}^k \int_{\R_+}|f_i(m_i,r_i)| dr_i\\
    & \leq 1.
\end{align*}
Hence the result holds by induction on $k$.
\end{proof}

\begin{proposition}\label{prop:Lipchitz}
The eigenvalue functions $\eigenfnn$ are Lipchitz with respect to the square-root metric $\rho_k$: There exists $A_\varphi>0$ such that $$|\eigenfnn(\vm)-\eigenfnn(\vm')|\leq A_\varphi\rho_k(\vm,\vm')$$
for all $\vm=(m_1,\dots,m_k),\vm'=(m'_1,\dots,m'_k)\in \N_0^k$.
\end{proposition}

\begin{proof}
Define $K:\N_0\times\R_+\rightarrow[0,\infty)$ by $K(m,r)=\frac{r^{m}}{m!}e^{-r}$. Then $\int_0^\infty K(m,r) dr=1$ and

\begin{align*}
    \eigenfn(\vm)= \int_{\R_+^k} a_\varphi (\sqrt{r_1},\dots,\sqrt{r_k})\prod_{i=1}^k K(m_i,r_i)dr_1\cdots dr_k
\end{align*}
where $\textbf{1}=(1,\dots,1)\in\N^k$. Notice that

\begin{align*}
    |\eigenfn(\vm) & -\eigenfn(\vm')| \\
    &\leq  \|a_\varphi\|_\infty \int_{\R_+^k}   \Big|\prod_{i=1}^k K(m_i,r_i) -\prod_{i=1}^k K(m'_i,r_i)\Big|dr_1\cdots dr_k  \\
    & \leq \|a_\varphi\|_\infty \sum_{i=1}^k \int_0^\infty |K(m_i,r_i)-K(m'_i,r_i)| dr_i \hspace{.7cm} \text{(by Lemma \ref{lem:intproduct})}\\
    &\leq 2\sqrt{\frac{2}{\pi}} \|a_\varphi\|_\infty\rho_k(\vm,\vm')
    \hspace{.5cm} \text{(by Lemma \ref{lem:Kbound})}.
\end{align*}

Hence by Corollary \ref{coro:eigenshift},
\begin{align*}
    |\eigenfnn(\vm)-\eigenfnn(\vm')| 
    &= |\eigenfn(\vm+\vn-\textbf{1}) -\eigenfn(\vm'+\vn-\textbf{1})|\\
    &\leq 2\sqrt{\frac{2}{\pi}} \|a_\varphi\|_\infty\rho_k(\vm+\vn-\textbf{1},\vm'+\vn-\textbf{1})\\
    &\leq 2\sqrt{\frac{2}{\pi}} \|a_\varphi\|_\infty \rho_k(\vm,\vm').
\end{align*}

Therefore the function $\eigenfnn$ are uniformly continuous with respect to the metric $\rho_k$, proving Theorem \ref{prop:inclusion}. \qedhere

\end{proof}

\section{Density of \texorpdfstring{$\fG_{\textbf{1}}$}{eigen functions} in \texorpdfstring{$\RO$}{slowly oscillating sequences}}\label{sec:density}

In the previous section we showed that $\specnn \subseteq \RO$. To show that the $C^*$-algebra generated by $\specnn$ is $\RO$, it is enough to show that $\specnn$ is dense in $\RO$ which will be proved in all generality in section \ref{sec:densitynn}. In this section we focus on the case $\vn =\textbf{1}=(1,\dots,1)\in\N^k$. In other words we will show that $\fG_{\textbf{1}}$, the set of eigenvalue functions we get for the case $\vn =\textbf{1}$, is dense in $\RO$. Note that in the above case, the Toeplitz operators under consideration are the Toeplitz operators on $\cF(\C^k)$ with seperately radial symbols.

This discussion is a generalization of the proof of density for the case $\vn=(1)$, presented in \cite{EM16}.

\subsection{Approximation by convolutions}\label{subsec:5.1}

The main goal of this subsection is to prove Proposition \ref{prop:tailbound}. For this we present a series of propositions and lemmas.\\

By a change of variable $(r_1,\dots, r_k)\rightarrow (r_1^2,\dots,r_k^2)$, we have 
\begin{align*}
    \eigenfn(\textbf{m}) &=  \int_{\R_+^k} a (r_1,\dots,r_k)  \prod_{i=1}^k g(m_i,r_i) dr_1\dots dr_k
\end{align*}
where $g(m,r)=\frac{2r^{2m+1}e^{-r^2}}{m!}$.
Let $h(x)=\sqrt{\frac{2}{\pi}}e^{-2x^2}$.

We have the following lemma from \cite{EM16} (Lemma 6.5).

\begin{lemma}\label{lem1} Let the functions $g$ and $h$ be as above. Then
$$\lim_{m\rightarrow\infty}\int_0^{\infty}|g(m,r)-h(\sqrt{m}-r)|dr=0.$$
\end{lemma}

Let $H(x_1,\dots,x_k)=\prod_{i=1}^kh(x_i)=(\frac{2}{\pi})^{\frac{k}{2}}e^{-2(x_1^2+\cdots +x_k^2)}$. 

\begin{lemma}\label{lem2}
Let $\epsilon>0$. Then there exists $N$ such that for all $m_i>N$, $i=1,\dots,k$
$$\int_{\R_+^k} \Big|\prod_{i=1}^kg(m_i,r_i)-H(\sqrt{m_1}-r_1,\dots, \sqrt{m_k}-r_k) \Big|dr_1\dots dr_k<\epsilon.$$
\end{lemma}

\begin{proof}
By Lemma \ref{lem1}, there exists $N$ such that for all $m>N$, $\int_0^{\infty}|g(m,r)-h(\sqrt{m}-r)|dr<\frac{\epsilon}{k}$. Also notice that $\int_0^{\infty}g(m,r)dr=1$ and $\int_0^{\infty}h(\sqrt{m}-r)dr\leq \int_{\R}h(\sqrt{m}-r)dr=\int_{\R}h(r)dr=1$ for any $m\in \N_0$.
Then if $m_i>N$, $i=1,\dots,k$,

\begin{align*}
    &\int_{\R_+^k} \Big|\prod_{i=1}^kg(m_i,r_i)-H(\sqrt{m_1}-r_1,\dots, \sqrt{m_k}-r_k)\Big|dr_1\dots dr_k\\
    &\hspace{4cm}=\int_{\R_+^k} \Big |\prod_{i=1}^kg(m_i,r_i)-\prod_{i=1}^kh(\sqrt{m_i}-r_i)\Big|dr_1\dots dr_k\\
    &\hspace{4cm}\leq \sum_{i=1}^k \int_0^\infty |g(m_i,r_i)-h(\sqrt{m_i}-r_i)| dr_i \hspace{.7cm} \text{(by Lemma \ref{lem:intproduct})}\\
    &\hspace{4cm}< k\frac{\epsilon}{k}=\epsilon. \qedhere
\end{align*}
\end{proof}

Let $f\in L^1(\R^k,dx)$ and $g\in L^\infty(\R^k,dx)$ where $dx$ denotes the Lebesgue measure on $\R^k$. Recall that the convolution of $f$ and $g$, denoted $f\ast g$, is given by
$$f\ast g(x)=g\ast f(x)=\int_{\R^k}f(x-y)g(y) dy,\ \ \ \ \ x \in \R^k.$$

The following proposition presents an approximation of the eigenvalue functions by convolutions.
\begin{proposition}\label{prop:eigen_convol} Let $a\in L^\infty(\R_+^k)$ and let $\epsilon>0$. Define $H\ast a$, by considering '$a$' as a function on $\R^k$ whose support is $\R_+^k$.
Then there exists $N$ such that for all $m_i>N$, $i=1,\dots, k$,
$$|\eigenfn(m_1,\dots,m_k)-(H\ast a)(\sqrt{m_1},\dots,\sqrt{m_k})|<\epsilon.$$
\end{proposition}

\begin{proof}
Since support of $a$ is $\R_+^k$,
\begin{align*}
    &|\eigenfn(m_1,\dots,m_k)-(H\ast a)(\sqrt{m_1},\dots,\sqrt{m_k})| \\
    &= \Big|\int_{\R_+^k} \Bigg (\prod_{i=1}^kg(m_i,r_i)-H(\sqrt{m_1}-r_1,\dots, \sqrt{m_k}-r_k)\Bigg ) a(y) dr_1\dots dr_k\Big|\\
    &\leq \|a\|_\infty\int_{\R_+^k}\Big|\prod_{i=1}^kg(m_i,r_i)-H(\sqrt{m_1}-r_1,\dots, \sqrt{m_k}-r_k)\Big|dr_1\dots dr_k.
\end{align*}
Hence the result follows from \ref{lem2}
\end{proof}

\begin{lemma}\label{lem:convol}
Let $b\in L^\infty(\R^k)$ and $a=\chi_{\R_+^k}b$. Let $\epsilon>0$. Then there exists $M$ such that for all $x_i>M$, $i=1,\dots,k$,
$$|H\ast a(x_1,\dots,x_k)-H\ast b(x_1,\dots,x_k)|<\epsilon.$$
\end{lemma}
\begin{proof}
Notice that for $x=(x_1,\dots,x_k)\in \R_+^k$,
\begin{align*}
    |H\ast a(x)-H\ast b(x)|
    &\leq \int_{\R^k}|a(y)-b(y)|H(x-y) dy\\
    &\leq \|b\|_\infty \int_{\R^k\setminus\R_+^k}H(x-y) dy\\
    &\leq \|b\|_\infty \sum_{j=1}^k \int_{\R_-}h(x_j-y_j)dy_j \prod_{\substack{i=1\\ i\neq j}}^k \int_{\R}h(x_i-y_i)dy_i\\
    &= \|b\|_\infty \sum_{j=1}^k \int_{\R_-}h(x_j-y_j)dy_j \\
    &= \|b\|_\infty \sum_{j=1}^k \int_{x_j}^\infty h(y_j) dy_j \hspace{.5 cm} (\text{by a change of variable}).
\end{align*}
The lemma holds as $\int_{x_j}^\infty h(y) dy$ approaches zero as $x_j$ goes to $\infty$.
\end{proof}

\begin{proposition}\label{prop:convol} 
Let $g\in C_{b,u}(\R^k)$ and let $\epsilon>0$. Then there is $b\in L^\infty(\R^k)$ such that 
$$\|H\ast b - g\|_\infty<\epsilon.$$
\end{proposition}

\begin{proof}
Let $\Tilde{h}$ denote the bump function (any compactly supported smooth function would suffice)
$$\Tilde{h}(\zeta)=e^{-\frac{1}{1-|\zeta|^2}} \text{\raisebox{2pt}{$\chi$}}_{B}(\zeta),\ \  \zeta\in\R^k$$
where $B$ is the open unit ball in $\R^k$.
Let $h$ denote the normalized Fourier inverse of $\Tilde{h}$, i.e., $$h=\frac{\mathscr{F}^{-1}(\Tilde{h})}{C}$$
where $C=\|\mathscr{F}^{-1}(\Tilde{h})\|_1$.
Define the approximate identity $h_t$ on $\R^k$ by 
$$h_t(x)=\cfrac{h(x/t)}{t}, \ \ x\in\R^k.$$
Notice that $\lim_{t\to 0}\|h_t\ast g-g\|_\infty=0$. Choose $h_{t_0}$ s.t. $\|h_{t_0}\ast g-g\|_\infty<\epsilon$.
Let $\hat{l}:=\frac{\mathscr{F}(h_{t_0})}{\mathscr{F}(H)}$. Then 
$$\hat{l}(\zeta)=\frac{\Tilde{h}(t_0\zeta)e^{|\zeta^2|/2}}{C}, \ \ \zeta\in\R^k.$$
Since $\hat{l}$ is a Schwartz function and $\mathscr{F}(h_{t_0})=\mathscr{F}(H)\hat{l}$, $h_{t_0}=H\ast l$ where $l=\mathscr{F}^{-1}(\hat{l})$.
Let $b=l\ast g$. Then $b\in L^\infty(\R^k)$ and $\|H\ast b-g\|_\infty=\|h_{t_0}\ast g-g\|_\infty<\epsilon$.
\end{proof}

Proposition \ref{prop:convol} can also be proved using Wiener's division lemma, as in  the proof of Proposition 5.4 in \cite{EM16}.

\begin{proposition}\label{prop:tailbound}
Let $\sigma\in \RO$ and let $\epsilon>0$. Then there exists $a\in L^\infty(\R_+^k)$ and $N$ such that for all $\vm=(m_1,\dots,m_k)\in \N_0^k$ with $m_i>N$, 
$$|\sigma(\vm)-\eigenfn(\vm)|<\epsilon.$$
\end{proposition}

\begin{proof}
This proof is similar to the proof of Proposition 6.8 in \cite{EM16}.

By Lemma \ref{lem:extend}, there exists $g\in C_{b,u}(\R_+^k)$ such that $g|_{\N_0^k}=\sigma$ and $\|g\|_\infty=\|\sigma\|_\infty$. Define $\Tilde{g}$ on $\R^k$ by $\Tilde{g}(x_1,\dots,x_k)=g(x_1^2,\dots,x_k^2)$. Then $\Tilde{g}\in C_{b,u}(\R^k)$. Then by Proposition \ref{prop:convol} there exists $b\in L^\infty(\R^k)$ such that 
$$\|H\ast b - \Tilde{g}\|_\infty<\frac{\epsilon}{3}.$$

Define $a$ on $\R_+^k$ by $a=b|_{\R_+^k}$. Then $a\in L^\infty(\R_+^k)$. Then by Proposition \ref{prop:eigen_convol} and Lemma \ref{lem:convol}, there exists $N_1,N_2$ such that for all $m_i>N_1$, $i=1,\dots,k$
$$|\eigenfn(m_1,\dots,m_k)-(H\ast a(\sqrt{m_1},\dots,\sqrt{m_k})|<\frac{\epsilon}{3}$$
and for all $m_i>N_2$,
$$|H\ast a(\sqrt{m_1},\dots,\sqrt{m_k})-H\ast b(\sqrt{m_1},\dots,\sqrt{m_k})|<\frac{\epsilon}{3}.$$
Then for all $\vm=(m_1,\dots,m_k)$ with $m_i>N=max\{N_1,N_2\}$
\begin{align*}
    &|\eigenfn(\vm)-\sigma(\vm)|\leq 
    |\eigenfn(\vm)-(H\ast a)(\sqrt{m_1},\dots,\sqrt{m_k})|\\
    &\hspace{3cm}+ |H\ast a(\sqrt{m_1},\dots,\sqrt{m_k})-H\ast b(\sqrt{m_1},\dots,\sqrt{m_k})|\\
    &\hspace{8cm}+\|H\ast b - \Tilde{g}\|_\infty\\
    &\hspace{3cm}<\frac{\epsilon}{3}+\frac{\epsilon}{3}+\frac{\epsilon}{3}=\epsilon. \qedhere
\end{align*}
\end{proof}

\subsection{The density of \texorpdfstring{$\fG_{\bf{1}}$}{G1} in \texorpdfstring{$\RO$}{Cbu}}
The proof of density (Theorem \ref{thm:density}) requires induction on $k$. Hence to indicate dependency on $k$, we identify $\fG_{\bf{1}}$ by $\fG_{\bf{1}}^k$ and $\eigenfn$ by $\eigenfb{a}^k$ as needed in this subsection.

\begin{lemma}\label{lem:ray}
Let $k>1$ and assume $\mathfrak{G}_{\bf{1}}^{k-1}$ is dense in $C_{b,u}(\N_0^{k-1},\rho_{k-1})$. For $i=1,\dots, k$ and $m_o\in \N_0$, define the set $K_i(m_0)\subset \N_0^k$ 
$$K_i(m_0)=\{\vm=(m_1,\dots,m_k) \in \N_0^k \ |\ m_i=m_0\}.$$
Let $\sigma\in \RO$ and let $\epsilon>0$. Then there exists $a\in L^\infty(\R_+^k)$ s.t. $$\|\sigma\chi_{K_i(m_0)}-\eigenfb{a}^k\|_\infty< \epsilon.$$ 
\end{lemma}

\begin{proof}
Identify $K_i(m_0)$ with $\N_0^{k-1}$ by the map $$p_{i,m_0}:(m_1,\dots, m_{i-1},m_o,m_{i+1}\dots,m_k)\mapsto (m_1,\dots, m_{i-1},m_{i+1},\dots,m_k).$$
Define $\sigma_{m_0}: \N_0^{k-1}\to \C$ by
$$\sigma_{m_0}(m_1,\dots,m_{k-1})=\sigma\chi_{K_i(m_0)}(p_{i,m_0}^{-1}(m_1,\dots,m_{k-1})).$$
Then $\sigma_{m_0}\in C_{b,u}(\N_0^{k-1},\rho_{k-1})$.
By the assumption there is $b\in L^\infty (\R_+^{k-1})$ s.t.
$$\|\sigma_{m_0}-\eigenfb{b}^{k-1}\|<\frac{\epsilon}{2}.$$
Since $\chi_{\{m_0\}}\in C_{b,u}(\N_0,\rho_{1})$, there is $c\in L^\infty (\R_+)$ s.t.
$$\|\chi_{\{m_0\}}-\eigenfb{c}\|<\|b\|_\infty \frac{\epsilon}{2}.$$
Now define $a\in L^\infty(\R_+^k)$ by $$a(\vx)=b(x_1,\dots,x_{i-1},x_{i+1},\dots,x_k)c(x_i) \quad \forall \ \vx \in \R_+^k.$$
Then $\eigenfb{a}^k(\vm)=\eigenfb{b}^{k-1}(p_i(\vm))\eigenfb{c}(m_i)$ for all $\vm \in \N_0^k$.
Let $\nu\in C_{b,u}(\N_0^{k},\rho_{k})$ be defined by $\nu(\vm)=\eigenfb{b}^{k-1}(p_i(\vm))\chi_{\{m_0\}}(m_i)$ for all $\vm \in \N_0^k$.
Then 
\begin{align*}
    \|\sigma\chi_{K_i(m_0)}-\eigenfb{a}^k\|_\infty &\leq \|\sigma\chi_{K_i(m_0)}-\nu\|_\infty+\|\nu-\eigenfb{a}^k\|_\infty \\
    &\leq \|\sigma_{m_0}-\eigenfb{b}^{k-1}\|_\infty + \|\eigenfb{b}^{k-1}\|_\infty\|\eigenfb{c}-\chi_{\{m_0\}}\|_\infty\\
    &< \frac{\epsilon}{2}+ \|b\|_\infty\frac{\epsilon}{2\|b\|_\infty}=\epsilon.
    \qedhere
\end{align*}
\end{proof}

\begin{lemma}\label{lem:boundK}
Let $k>1$ and assume $\mathfrak{G}_{\bf{1}}^{k-1}$ is dense in $C_{b,u}(\N_0^{k-1},\rho_{k-1})$. Let $K_N\subset \N_0^k$ be defined by 
$$K_N:=\{\vm=(m_1,\dots,m_k)\in \N_0^k| m_i\leq N \text{ for some } i \}$$
Let $\sigma\in \RO$ and let $\epsilon>0$. Then there exists $a\in L^\infty(\R_+^k)$ s.t. $$\|\sigma\chi_{K_N}-\eigenfb{a}\|_\infty< \epsilon.$$
\end{lemma}

\begin{proof}
For $i=1,\dots,k$ and $j=0,1,\dots,N$, let the sets $K_i(j)$ be as defined in Lemma \ref{lem:ray}.
Then there exists $\sigma_{(i,j)}\in C_{b,u}(\N_0^{k},\rho_{k})$ s.t. $\sigma_{(i,j)}\chi_{K_i(j)}=\sigma_{(i,j)}$ and 
$$\sigma\chi_{K_N}= \sum_{i=1}^k \sum_{j=1}^N\sigma_{(i,j)}.$$ 
The functions $\sigma_{(i,j)}$ can be constructed by considering $\sigma\chi_{K_i(j)}$ and removing up to finitely many points from its support. By Lemma \ref{lem:ray}, there exists $a_{(i,j)}\in L^\infty(\R_+^k)$ s.t. 
$$\|\sigma_{(i,j)}-\eigenfb{a_{(i,j)}}\|_\infty< \frac{\epsilon}{kN}.$$
Define $a\in L^\infty(\R_+^k)$ by $a=\sum_{i=1}^k \sum_{j=1}^N a_{(i,j)}$. Then
$\eigenfb{a}=\sum_{i=1}^k \sum_{j=1}^N \eigenfb{a_{(i,j)}}$ and
\begin{align}
    \|\sigma\chi_{K_N}-\eigenfb{a}\|_\infty & \leq \sum_{i=1}^k \sum_{j=1}^N \|\sigma_{(i,j)}-\eigenfb{a_{(i,j)}}\|_\infty\\
    &< kN \frac{\epsilon}{kN}=\epsilon. \qedhere
\end{align}
\end{proof}

\begin{theorem}\label{thm:density}
$\fG_{\textbf{1}}$ is dense in $\RO$.
\end{theorem}

\begin{proof}
The proof is by induction on $k$. The result for $k=1$ was proved in \cite{EM16}. Assume the result is true for $k-1$.
Let $\sigma\in \RO$ and let $\epsilon>0$. By Proposition \ref{prop:tailbound} there is $N$ and $b\in L^\infty(\R_+^k)$ s.t. for all $\vm\in \N_0^k$ with $m_i>N$,
$$|\sigma(\vm)-\eigenfb{b}(\vm)|<\frac{\epsilon}{2}.$$

By the induction hypothesis and Lemma \ref{lem:boundK}, there exists $c\in L^\infty(\R_+^k)$ s.t. 
$$\| (\sigma-\eigenfb{b})\chi_{K_N} -\eigenfb{c}\|<\frac{\epsilon}{2}.$$
Let $a=b+c$. Then $\eigenfb{a}=\eigenfb{b}+\eigenfb{c}$ and
\begin{align*}
    \|\sigma- \eigenfb{a}\|_\infty &= \|(\sigma-\eigenfb{b})\chi_{K_N^c}+ (\sigma-\eigenfb{b})\chi_{K_N}-\eigenfb{c}\|_\infty\\
    & \leq \|(\sigma-\eigenfb{b})\chi_{K_N^c}\|_\infty+\|(\sigma-\eigenfb{b})\chi_{K_N}-\eigenfb{c}\|_\infty \\
    &< \sup_{\substack{\vm\in \N_0^k\\ m_i>N \ \forall i}} |\sigma(\vm)-\eigenfb{b}(\vm)|+ \frac{\epsilon}{2}\\
    &<\frac{\epsilon}{2} + \frac{\epsilon}{2} =\epsilon. \qedhere
\end{align*}
\end{proof}

\section{The \texorpdfstring{$C^*$}{C*}-algebra generated by \texorpdfstring{$\specnn$}{G}}\label{sec:densitynn}

In this section we prove the density theorem in all generality: the $C^*$-algebra generated by $\specnn$, where $\vn=(n_1,\dots,n_k)$, is the $C^*$-algebra $\RO$.
We already noticed that $\eigenfnn\in \RO$ in Proposition \ref{prop:inclusion}.
Recall the shift operators in definition \ref{def:shifts}.
Following lemma is a restatement of Corollary \ref{coro:eigenshift}
\begin{lemma}\label{lem:eigenshift} The eigenvalue function $\eigenfnn$ is the $(\vn-\textbf{1})^{th}$ left shift of $\eigenfn$, i.e.,
$$\eigenfnn=\tau_L^{\textbf{n-1}}\eigenfn.$$
\end{lemma}

\begin{proposition}
The set of eigenvalue functions $\specnn$ is dense in the $C^*$-algebra $\RO$ for any $\vn=(n_1,\dots,n_k)\in \N_0^k$.
\end{proposition}

\begin{proof}
Let $\sigma\in \RO$ and let $\epsilon>0$. By Lemma \ref{lem:rshiftop}, $\sigma' := \tau_R^{\vn-\textbf{1}}\sigma\in \RO$. Also $\fG_{\textbf{1}}$ is dense in $ \RO$ by Theorem \ref{thm:density}. Hence there exists $a\in L^{\infty}(\mathbb{R}_+^k)$ such that $\|\sigma'-\gamma_{\textbf{1},a}\|_\infty<\epsilon$. Then
\begin{align*}
    \|\sigma-\gamma_{\vn,a}\|_\infty &= \|\tau_L^{\vn-\textbf{1}}\sigma'-\tau_L^{\vn-\textbf{1}}\gamma_{\textbf{1},a}\|_\infty \\
    &= \|\tau_L^{\vn-\textbf{1}}\|\|\sigma'-\gamma_{\textbf{1},a}\|_\infty<\epsilon.
\end{align*}
as $ \|\tau_L^{n-1}\|<1$.\qedhere
\end{proof}

It follows that the $C^*$-algebra generated by $\specnn$ is the  $C^*$-algebra $\RO$.


\section{A comparison of \texorpdfstring{$\RO$}{Cbu} with \texorpdfstring{$\tensorRO$}{tensorproduct algebra}}

In this section we compare the $C^*$-algebra  $C_{b,u}(\mathbb{N}_0^k, \rho_k)$ with the $C^*$-tensor product $\tensorRO$.

By Corollary \ref{coro:eigenproduct} and the density of $\mathfrak{G}_{\textbf{1}}$ in $\ro$, we have the following proposition.

\begin{proposition}
Let $$\Tilde{\mathfrak{G}}_{\vn}=\{\eigenfnn | a\in L^\infty(\R_+)\otimes\cdots\otimes L^\infty(\R_+)\}.$$
Then $\Tilde{\mathfrak{G}}_{\vn}$ is dense in $\tensorRO$.
\end{proposition}

We have the inclusions 
$$C_{b,u}(\mathbb{N}_0^{k-1}, \rho_{k-1})\otimes C_{b,u}(\mathbb{N}_0, \rho_1)\xhookrightarrow{i_k}
C_{b,u}(\mathbb{N}_0^k, \rho_k).$$
In fact 
$$\tensorRO \xhookrightarrow{}
C_{b,u}(\mathbb{N}_0^k, \rho_k).$$

However, the above inclusions are not necessarily isomorphisms. We present a counter example for $k=2$.

\subsection{A counter example}\label{subsec:counterex}
Here we construct a counter example to show that $C_{b,u}(\mathbb{N}_0^2, \rho_2)$ is strictly larger than $C_{b,u}(\mathbb{N}_0, \rho_1)\otimes C_{b,u}(\mathbb{N}_0, \rho_1)$.
Let 
$$I_{i}:[\sqrt{i},\sqrt{i}+\pi)\ \ ; \ i\in \N_0.$$
Define $g:\N_0^2\to \C$ by $$g(i,j)=\sin(\sqrt{j}-\sqrt{i}) \mbox{\Large$\chi$}_{I_{i}}(\sqrt{j})\ \ ; \ (i,j)\in \N_0^2.$$

\begin{lemma}\label{lem:incbu}
Let $g$ be the function defined above. Then $g\in C_{b,u}(\mathbb{N}_0^2, \rho_2)$.
\end{lemma}

\begin{proof}
Clearly $g$ is bounded. First we show that if $|\sqrt{j}-\sqrt{j'}|<\pi$,
$$|g(i,j)-g(i,j')|<|\sqrt{j}-\sqrt{j'}|.$$
If $\sqrt{j},\sqrt{j'}\notin I_i$, $$|g(i,j)-g(i,j')|=0<|\sqrt{j}-\sqrt{j'}|.$$
W.l.o.g. assume $\sqrt{j}\in I_i$. Then $\sqrt{j'}\in(\sqrt{i}-\pi,\sqrt{i}+2\pi)$ as $|\sqrt{j}-\sqrt{j'}|<\pi$. Note that if $\sqrt{j'}\in I_i$, $\sin(\sqrt{j'}-\sqrt{i})=g(j',i)$  and if $\sqrt{j'}\in (\sqrt{i}-\pi,\sqrt{i})\cup (\sqrt{i}+\pi,\sqrt{i}+2\pi)$, $\sin(\sqrt{j'}-\sqrt{i})<0$. Hence
\begin{align*}
    |g(i,j)-g(i,j')|&\leq |\sin(\sqrt{j}-\sqrt{i})-\sin(\sqrt{j'}-\sqrt{i})|\\
    &\leq |\sqrt{j}-\sqrt{j'}|
\end{align*}
as required. Next we prove that if $|\sqrt{i}-\sqrt{i'}|<\pi$, $$|g(i,j)-g(i',j)|<|\sqrt{i}-\sqrt{i'}|.$$
If $\sqrt{j}\notin I_{i}\cup I_{i'}$, $|g(i,j)-g(i',j)|=0<|\sqrt{i}-\sqrt{i'}|$. W.l.o.g., assume that $\sqrt{j}\in I_i$. Then $\sqrt{j}\in(\sqrt{i'}-\pi,\sqrt{i'}+2\pi)$ as $|\sqrt{i}-\sqrt{i'}|<\pi$. Therefore 
\begin{align*}
    |g(i,j)-g(i',j)|&\leq |\sin(\sqrt{j}-\sqrt{i})-\sin(\sqrt{j}-\sqrt{i'})|\\
    &\leq |\sqrt{i}-\sqrt{i'}|,
\end{align*}
proving the inequality.\\
Now notice that for all $(i,j),(i',j')\in \N_0^2$ such that $\rho_2((i,j),(i',j'))<\pi$,
\begin{align*}
    |g(i,j)-g(i',j')| &\leq |g(i,j)-g(i',j)|+ |g(i',j)-g(i',j')|\\
    &\leq |\sqrt{i}-\sqrt{i'}| +|\sqrt{j}-\sqrt{j'}|\\
    &< \rho_2((i,j),(i',j')),
\end{align*}
proving that $g\in C_{b,u}(\mathbb{N}_0^2, \rho_2)$.
\end{proof}

\begin{definition}
If $X$ is a locally compact Hausdorff space and if $A$ is a $C^*$-algebra equipped with $\|\cdot\|_A$, let $C^b(X,A)$ be the set of all continuous bounded functions $f:X\to A$ equipped with the norm $\|\cdot\|_\infty$ given by
$$\|f\|_\infty=\sup_{x\in X} \|f(x)\|_A.$$
\end{definition}

It is well known that $C^b(X,A)$ is $C^*$-algebra.
Recall that a subset of a topological space is said to be precompact if its closure is compact.
The following theorem from Williams \cite{W03} describes a criterion to check whether an element in the C$^*$-algebra $C^b(X,A)$ belongs to the possibly smaller C$^*$-algebra $C^b(X)\otimes A$. We will use the Theorem \ref{thm:will} to show that $g\notin C_{b,u}(\mathbb{N}_0, \rho_1)\otimes C_{b,u}(\mathbb{N}_0, \rho_1)$.

\begin{theorem}[Williams  \cite{W03}]\label{thm:will}
If $X$ is a locally compact Hausdorff space and if A is a $C^*$-algebra,
then $f\in C^b(X,A)$ is in $C^b(X)\otimes A$ if and only if the range of $f$, $R(f) := \{f(x) :
x \in X\}$, is precompact.
\end{theorem}

In order to use Theorem \ref{thm:will} we present several lemmas about $g$. The proof of Lemma \ref{lem:cb} is trivial as $\N_0$ has discrete topology and $g$ is bounded.

\begin{lemma}\label{lem:cb}
Let $g$ be the function defined above. Define $\Tilde{g}:\N_0\to C_{b,u}(\mathbb{N}_0, \rho_1)$ by
$$\Tilde{g}(i)=g(i,\cdot).$$
Then $\Tilde{g}\in C^{b}(\mathbb{N}_0, C_{b,u}(\mathbb{N}_0, \rho_1))$.
\end{lemma}

\begin{lemma}
Let $\Tilde{g}:\N_0\to C_{b,u}(\mathbb{N}_0, \rho_1)$ be defined as in Lemma \ref{lem:cb}. Then $\|\Tilde{g}(i)\|_\infty\geq \frac{1}{2}$ for all $i\in\N_0$.
\end{lemma}

\begin{proof}
Notice that for all $i\in \N_0$ and for all $p\in \N$,
\begin{align*}
    \sqrt{i+p}-\sqrt{i+p-1} \leq 1 <\frac{2\pi}{3}
\end{align*}
because $\sqrt{i+p}-\sqrt{i+p-1}$ attains its maximum when $i+p=1$. Fix $i\in \N_0$. Then the sequence $\{\sqrt{i+p}\}_{p=1}^\infty$ contains a point in any interval of length $\frac{2\pi}{3}$ and, in particular, it contains a point in $(\sqrt{i}+\frac{\pi}{6}, \sqrt{i}+\frac{5\pi}{6})$. Denote that point by $\sqrt{i+p_0}$. Then $\sqrt{i+p_0}-\sqrt{i}\in (\frac{\pi}{6},\frac{5\pi}{6})$ and
\begin{align*}
    \|\Tilde{g}(i)\|_\infty &\geq |(\Tilde{g}(i))(i+p_0)|\\
    &= \sin (\sqrt{i+p_0}-\sqrt{i})\\
    &\geq \frac{1}{2}.
\end{align*}
\end{proof}

\begin{lemma}\label{lem:precompact}
Let $\Tilde{g}:\N_0\to C_{b,u}(\mathbb{N}_0, \rho_1)$ be defined as in Lemma \ref{lem:cb}. Then the range of $\Tilde{g}$ is not precompact.
\end{lemma}

\begin{proof}
Whenever $|\sqrt{i_1}-\sqrt{i_2}|>\pi$, $I_{i_1}\cap I_{i_2}=\emptyset$ and hence
\begin{align*}
    \|\Tilde{g}(i_1)-\Tilde{g}(i_2)\|_\infty = \max \{\|\Tilde{g}(i_1)\|_\infty,\|\Tilde{g}(i_2)\|_\infty \} \geq \frac{1}{2}.
\end{align*}
Thus there exits a sequence $\{ \Tilde{g}(i_s)\}_{s=1}^\infty$ s.t. $\|\Tilde{g}(i_{s_1})-\Tilde{g}(i_{s_2})\|_\infty\geq \frac{1}{2}$ whenever $s_1\neq s_2$. It follows that the range of $\Tilde{g}$ is not totally bounded and hence it is not precompact.
\end{proof}

By Lemma \ref{lem:cb}, Lemma \ref{lem:precompact}, and Therorem \ref{thm:will}, we have that $g\notin C_{b,u}(\mathbb{N}_0, \rho_1)\otimes C_{b,u}(\mathbb{N}_0, \rho_1)$. Also by Lemma \ref{lem:incbu}, $g\in C_{b,u}(\mathbb{N}_0^2, \rho_2)$. Hence we have the following proposition.
\begin{proposition}
$C_{b,u}(\mathbb{N}_0^2, \rho_2)$ is strictly larger than $C_{b,u}(\mathbb{N}_0, \rho_1)\otimes C_{b,u}(\mathbb{N}_0, \rho_1)$.
\end{proposition}


\section{Appendix: The proof of Lemma \ref{lem:extend}}
In this appendix, we present the proof of Lemma \ref{lem:extend}. It is easy to see that $f|_{\N_0^k}=\sigma$ , from the definition of $f$. 

\subsection{The uniform norm of $f$ }
To show that $\|f\|_{\infty}=\|\sigma\|_\infty$, we will use the following lemma.

\begin{lemma}\label{lem:extend_bound_1}
Let $s\in \N$ and let $a_i\in(0,1)$ for all $i=1,\dots,s$ Then 
$$1+\sum_{l=1}^s(-1)^l \sum_{\substack{i_1,\dots,i_l\in \{1,\dots,s\}\\ i_1<\cdots <i_l}}\prod_{q=1}^l a_{i_q}\geq 0$$
\end{lemma}

\begin{proof}
Notice that the statement is true for $s=1$. Assume the result is true for $s$. Then
\begin{align*}
    1+\sum_{l=1}^{s+1}(-1)^l \sum_{\substack{i_1,\dots,i_l\in \{1,\dots,s+1\}\\ i_1<\cdots <i_l}}\prod_{q=1}^l a_{i_q} &=
    1+\sum_{l=1}^s(-1)^l \sum_{\substack{i_1,\dots,i_l\in \{1,\dots,s\}\\ i_1<\cdots <i_l}}\prod_{q=1}^l a_{i_q}\\
    &\hspace{1cm}- a_{s+1}\Bigg (1+\sum_{l=1}^s(-1)^l \sum_{\substack{i_1,\dots,i_l\in \{1,\dots,s\}\\ i_1<\cdots <i_l}}\prod_{q=1}^l a_{i_q}\Bigg )\\
    &= (1-a_{s+1})\Bigg ( 1+\sum_{l=1}^s(-1)^l \sum_{\substack{i_1,\dots,i_l\in \{1,\dots,s\}\\ i_1<\cdots <i_l}}\prod_{q=1}^l a_{i_q} \Bigg )\\
    &\geq 0.\qedhere
\end{align*}
Hence the statement is true by induction on $s$.
\end{proof}

Note that $f_{\vm}$ can also be written as
$$f_\vm(\vx)= \sigma(
\vm)B_0(\vx,\vm)+ \sum_{\substack{j_1,\dots,j_s\in\{1,\dots,k\}\\ j_1<\cdots <j_s}} \sigma(\vm+e_{j_1}+\cdots + e_{j_s})B_{j_1,\dots,j_s}(\vx,\vm)$$
where 
$$B_{j_1,\dots,j_s}(\vx,\vm)=\prod_{p=1}^s\frac{\sqrt{x_{j_p}}-\sqrt{m_{j_p}}}{\sqrt{m_{j_p}+1}-\sqrt{m_{j_p}}} \Bigg(1+ \sum_{l=1}^{k-s} (-1)^l\sum_{\substack{i_1,\dots,i_l\in \{1,\dots,k\}\setminus\{j_1,\dots,j_s\}\\ i_1<\cdots <i_l}}\prod_{q=1}^l \frac{\sqrt{x_{i_q}}-\sqrt{m_{i_q}}}{\sqrt{m_{i_q}+1}-\sqrt{m_{i_q}}}\Bigg )$$ and
$$B_0(\vx,\vm)=1+ \sum_{l=1}^{k}(-1)^l \sum_{\substack{i_1,\dots,i_l\in \{1,\dots,k\}\\ i_1<\cdots <i_l}}\prod_{q=1}^l \frac{\sqrt{x_{i_q}}-\sqrt{m_{i_q}}}{\sqrt{m_{i_q}+1}-\sqrt{m_{i_q}}}.$$
The coefficients $B_{j_1,\dots,j_s}(\vx,\vm)$ are computed by summing the products 
$$(-1)^{l-s}\prod_{q=1}^l\frac{\sqrt{x_{j_q}}-\sqrt{m_{j_q}}}{\sqrt{m_{j_q}+1}-\sqrt{m_{j_q}}}$$
over $\{i_1,\dots,i_l\}$ such that $i_1<\cdots<i_l$ and  $\{j_1,\dots,j_s\}\subset \{i_1,\dots,i_l\}$.

Consider the sum $$B_0(\vx,\vm)+ \sum_{\substack{j_1,\dots,j_s\in\{1,\dots,k\}\\ j_1<\cdots <j_s}} B_{j_1,\dots,j_s}(\vx,\vm).$$
Let $l\in \N$ and let $i_1,\dots, i_l\in \{1,\dots,k\}$ such that $i_1<\cdots<i_l$. Note that $(-1)^{l-s}\prod_{q=1}^l\frac{\sqrt{x_{i_q}}-\sqrt{m_{i_q}}}{\sqrt{m_{i_q}+1}-\sqrt{m_{i_q}}}$ is a term in $B_{j_1,\dots,j_s}(\vx,\vm)$ whenever $\{j_1,\dots,j_s\}\subset \{i_1,\dots,i_l\}$.
Therefore in the above sum, the coefficient of $\prod_{q=1}^l\frac{\sqrt{x_{i_q}}-\sqrt{m_{i_q}}}{\sqrt{m_{i_q}+1}-\sqrt{m_{i_q}}}$ is given by
$$\sum_{s=0}^l (-1)^{l-s}{l \choose s}=(1-1)^l=0.$$

Hence 
$$B_0(\vx,\vm)+\sum_{\substack{j_1,\dots,j_s\in\{1,\dots,k\}\\ j_1<\cdots <j_s}} B_{j_1,\dots,j_s}(\vx,\vm)=1.$$

Also, if $\vx\in \prod_{i=1}^k[m_i,m_i+1)$, $B_{j_1,\dots,j_s}(\vx,\vm)\geq 0$ by Lemma \ref{lem:extend_bound_1} and hence 
\begin{align*}
    |f_\vm(\vx)| & \leq \|\sigma\|_\infty \Bigg( B_0(\vx,\vm)+ \sum_{\substack{j_1,\dots,j_s\in\{1,\dots,k\}\\ j_1<\cdots j_s}} B_{j_1,\dots,j_s}(\vx,\vm) \Bigg )\\
    &=  \|\sigma\|_\infty.
\end{align*}
Therefore $\|f\|_\infty=\|\sigma\|_\infty$ as $f|_{\N_0^k}=\sigma$. 

\subsection{Some useful lemmas}

The following lemma is quite useful in the proofs that follow.
\begin{lemma}\label{lem:extend_1}
Let $l,s\in \{1,\dots,k\}$. Suppose $i_1,\dots,i_l\in\{1,\dots,k\}$ such that $i_1<\dots<i_l$. Then
$$a^{l+1}_{i_1,\dots,s,\dots,i_{l}}(\vm)=a^l_{i_1,\dots,i_l}(\vm+e_{s})-a^l_{i_1,\dots,i_l}(\vm)$$
\end{lemma}
\begin{proof}
To keep the notations simple, we will assume $s>i_l$ and we label $s$ by $i_{l+1}$.
\begin{align*}
    a^{l+1}_{i_1,\dots,i_{l+1}}(\vm) &= \sum_{q=0}^{l+1} (-1)^{l+1-q} \sum_{\substack{j_1,\dots,j_{q}\in \{i_1,\dots,i_{l+1}\} \\ j_1<\cdots<j_{q}}} \sigma(\vm +e_{j_1}+\cdots + e_{j_q})\\
    &= (-1)^{l+1} \sigma(\vm) + \sum_{q=1}^{l} (-1)^{l+1-q}\Bigg ( \sum_{\substack{j_1,\dots,j_{q}\in \{i_1,\dots,i_{l}\} \\ j_1<\cdots<j_{q}}} \sigma(\vm +e_{j_1}+\cdots +e_{j_q})\\ 
    &\hspace{1cm}   +\sum_{\substack{j_1,\dots,j_{q-1}\in \{i_1,\dots,i_{l}\} \\ j_1<\cdots<j_{q-1}}} \sigma(\vm  +e_{j_1}+\cdots + e_{j_{q-1}} +e_{i_l}) \Bigg ) + \sigma(\vm+e_{i_1}+\cdots +e_{i_l})\\
    &= -\sum_{q=0}^{l} (-1)^{l-q} \sum_{\substack{j_1,\dots,j_{q}\in \{i_1,\dots,i_{l}\} \\ j_1<\cdots<j_{q}}} \sigma(\vm +e_{j_1}+\cdots +e_{j_q}) \\
    &\hspace{2cm} + \sum_{q=1}^{l+1} (-1)^{l+1-q} \sum_{\substack{j_1,\dots,j_{q-1}\in \{i_1,\dots,i_{l}\} \\ j_1<\cdots<j_{q-1}}} \sigma(\vm + e_{i_l} +e_{j_1}+\cdots +e_{j_{q-1}})\\
    &=-a^l_{i_1,\dots,i_l}(\vm) + \sum_{q=0}^{l} (-1)^{l-q} \sum_{\substack{j_1,\dots,j_{q}\in \{i_1,\dots,i_{l}\} \\ j_1<\cdots<j_{q}}} \sigma((\vm +e_{i_l}) +e_{j_1}+\cdots +e_{j_q})\\
    &= a^l_{i_1,\dots,i_l}(\vm+e_{l+1})-a^l_{i_1,\dots,i_l}(\vm).
\end{align*}
\end{proof}

Notice that $f$ is defined on $\prod_{i=1}^k[m_i,m_i+1)$ by $f_\vm$. But the following lemma allows $f$ to be defined on $\prod_{i=1}^k[m_i,m_i+1]$ by $f_\vm$.

\begin{lemma}
Let $\vx\in \R_+^k$, $\vm\in \N_0^k$ and let $s\in \{1,\dots, k\}$. Suppose $\lfloor x_i \rfloor=m_i$ for all $i$. Then $$f(\vx+ (1-x_{s})e_{s})=f_{\vm}(\vx+ (1-x_{s})e_{s}).$$

\end{lemma}

\begin{proof}
Note that the $s^{th}$ coordinate of $\vx+ (1-x_{s})e_{s}$ is $m_s+1$.
Therefore 
\begin{align*}
    f_{\vm}(\vx+ (1-x_{s})e_{s}) & = \sigma(\vm) + \sum_{l=1}^{k-1}  \sum_{\substack{i_1,\dots,i_l\in\{1,\dots,k\}\setminus \{s\} \\ i_1<\cdots<i_l}} a^l_{i_1,\dots,i_l}(\vm)\prod_{q=1}^l\frac{\sqrt{x_{i_q}}-\sqrt{m_{i_q}}}{\sqrt{m_{i_q}+1}-\sqrt{m_{i_q}}} \\
    & \hspace{1cm}+ a^1_{s}+ \sum_{l=2}^{k} \sum_{\substack{i_1,\dots,i_l\in\{1,\dots,k\} \\ s\in \{i_1,\dots,i_k\}\\ i_1<\cdots<i_l}} a^l_{i_1,\dots,i_l}(\vm)\prod_{\substack{q=1\\ i_q\neq s}}^l\frac{\sqrt{x_{i_q}}-\sqrt{m_{i_q}}}{\sqrt{m_{i_q}+1}-\sqrt{m_{i_q}}}\\
    &= \sigma(\vm+e_s)+ \Bigg (\sum_{l=1}^{k-1} \sum_{\substack{i_1,\dots,i_l\in\{1,\dots,k\}\setminus \{s\} \\ i_1<\cdots<i_l}} (a^l_{i_1,\dots,i_l}(\vm)+a^{l+1}_{i_1,\dots,s,\dots,i_{l}}(\vm))\\
    &\hspace{2cm}\times\prod_{q=1}^l\frac{\sqrt{x_{i_q}}-\sqrt{m_{i_q}}}{\sqrt{m_{i_q}+1}-\sqrt{m_{i_q}}}\Bigg) 
\end{align*}
by reindexing the last summand and because $\sigma(\vm)+a^1_s=\sigma(\vm+e_s)$. Also,
$$a^l_{i_1,\dots,i_l}(\vm)+a^{l+1}_{i_1,\dots,s,\dots,i_{l}}(\vm) = a^l_{i_1,\dots,i_l}(\vm+e_s) $$
by Lemma \ref{lem:extend_1}. Then
\begin{align*}
    f_{\vm}(\vx+ (1-x_{s})e_{s}) & = \sigma(\vm+e_s) + \sum_{l=1}^{k-1}  \sum_{\substack{i_1,\dots,i_l\in\{1,\dots,k\}\setminus\{s\} \\ i_1<\cdots<i_l}} a^l_{i_1,\dots,i_l}(\vm+e_s)\prod_{q=1}^l\frac{\sqrt{x_{i_q}}-\sqrt{m_{i_q}}}{\sqrt{m_{i_q}+1}-\sqrt{m_{i_q}}}\\
    &=f_{\vm+e_s}(\vx+ (1-x_{s})e_{s})\\
    &=f(\vx+ (1-x_{s})e_{s}).
\end{align*}
The second equality in the above computation holds as the $s^{th}$ coordinate of $\vx+ (1-x_{s})e_{s}$ is $m_s+1$ and hence any term in $f_{m+e_s}(\vx+(1-x_s)e_s)$ indexed by $\{i_1,\dots ,i_l\}$ would vanish if $s\in\{i_1,\dots,i_l\}$.
\end{proof}

As a consequence of above lemma, we have the following corollary.
\begin{corollary}\label{lem:extend_4}
Let $\vm\in \N_0^k$ and let $\vx\in\prod_{i=1}^k[m_i,m_i+1]$. Then $$f(\vx)=f_{\vm}(\vx).$$
\end{corollary}

\begin{lemma}\label{lem:extend_2}
Let $l\in \N$ and let $a_i,b_i \in \C$ such that $|a_i|,|b_i|\leq 1$ for all $i=1,\dots, l$. Then 
$$\Bigg | \prod_{i=1}^la_i-\prod_{i=1}^lb_i \Bigg | \leq \sum_{i=1}^l|a_i-b_i|.$$
\end{lemma}

\begin{proof}
Clearly, the result holds for $l=1$. Assume that the result is true for $l$. Then by triangle inequality,
\begin{align*}
    \Big | \prod_{i=1}^{l+1}a_i-\prod_{i=1}^{l+1}b_i \Big | &\leq
    |a_{l+1}-b_{l+1} | \Big|\prod_{i=1}^l a_i\Big| +  | b_{l+1}| \Big | \prod_{i=1}^la_i-\prod_{i=1}^lb_i  \Big |\\
    &\leq |a_{l+1}-b_{l+1} | + \Big | \prod_{i=1}^la_i-\prod_{i=1}^lb_i \Big |\\
    & \leq \sum_{i=1}^{l+1} |a_i-b_i|.
\end{align*}
Hence, the result is true by induction on $l$.

\end{proof}

\subsection{The uniform continuity of $f$ with respect to the square-root metric}
Let $\vx,\vx'\in\R_+^k$ s.t. $\rho_k(\vx,\vx')<\delta$. Let $\vm=(m_1,\dots,m_k)$ and $\vm'=(m'_1,\dots,m'_k)$ where $m_i=\lfloor x_i \rfloor $ and $m'_i=\lfloor x'_i\rfloor$ for $i=1,\dots,k$.\\
Case I. Assume $x'_i\in [m_i,m_{i+1}]$ for all $i$. Then by Corollary \ref{lem:extend_4}

\begin{align*}
    |f(\vx)-  f(\vx')| &\leq \sum_{l=1}^k\sum_{\substack{i_1,\dots,i_l\in\{1,\dots,k\} \\ i_1<\cdots<i_l}} |a^l_{i_1,\dots,i_l}(\vm)|\\
    & \hspace{1cm} \times \Bigg|\prod_{q=1}^l\frac{\sqrt{x_{i_q}}-\sqrt{m_{i_q}}}{\sqrt{m_{i_q}+1}-\sqrt{m_{i_q}}}-\prod_{q=1}^l\frac{\sqrt{x'_{i_q}}-\sqrt{m_{i_q}}}{\sqrt{m_{i_q}+1}-\sqrt{m_{i_q}}}\Bigg|\\
    &\leq 
    \sum_{l=1}^k\sum_{\substack{i_1,\dots,i_l\in\{1,\dots,k\} \\ i_1<\cdots<i_l}} |a^l_{i_1,\dots,i_l}(\vm)| \sum_{q=1}^l\frac{\big |\sqrt{x_{i_q}}-\sqrt{x'_{i_q}}\big |}{\sqrt{m_{i_q}+1}-\sqrt{m_{i_q}}}
\end{align*}
by Lemma \ref{lem:extend_2}. 

Fix $l\in \{1,\dots,k\}$ and $i_1,\dots,i_l\in\{1,\dots,k\}$ s.t. $i_1<\cdots<i_l$.\\
If $\sqrt{m_{i_q}+1}-\sqrt{m_{i_q}}\geq \sqrt{\delta}$ for all $q=1,\dots,l$,

\begin{align*}
        |a^l_{i_1,\dots,i_l}(\vm)| \sum_{q=1}^l\frac{|\sqrt{x_{i_q}}-\sqrt{x'_{i_q}}|}{\sqrt{m_{i_q}+1}-\sqrt{m_{i_q}}} 
        &\leq |a^l_{i_1,\dots,i_l}(\vm)| \frac{l\delta}{\sqrt{\delta}}\\
     &\leq 2^ll\|\sigma\|_\infty \sqrt{\delta}
\end{align*}
as $|a^l_{i_1,\dots,i_l}(\vm)|\leq 2^l\|\sigma\|_\infty $.

Assume  $\sqrt{m_{i_q}+1}-\sqrt{m_{i_q}}<\sqrt{\delta}$ for some $q\in \{1,\dots,l\}$. W.l.o.g. assume $\sqrt{m_{i_l}+1}-\sqrt{m_{i_l}}<\sqrt{\delta}$. 

If $l\geq 2$, 
\begin{align*}
    |a^l_{i_1,\dots,i_l}&(\vm)|\\
    &= |a^{l-1}_{i_1,\dots,i_l}(\vm+e_{i_l})-a^{l-1}_{i_1,\dots,i_l}(\vm)|\\
    &\leq \sum_{\substack{j_1,\dots,j_q\in\{i_1,\dots,i_{l-1}\} \\ q\in\{0,\dots,l-1\}}}
    |\sigma(\vm+ e_{j_1}+\cdots+e_{j_q}+e_{i_l})-\sigma(\vm+ e_{j_1}+\cdots+e_{j_q})|\\
    &\leq 2^{l-1}\omega_{\rho_k,\sigma}(\sqrt{\delta}).
\end{align*}

Hence
\begin{align*}
    |f(\vx)-  f(\vx')| &\leq \sum_{l=1}^k\sum_{\substack{i_1,\dots,i_l\in\{1,\dots,k\} \\ i_1<\cdots<i_l}} 2^{l-1}l \max \{2\|\sigma\|_\infty \sqrt{\delta},  \omega_{\rho_k,\sigma}(\sqrt{\delta})\} \\
    & = A_k\max \{2\|\sigma\|_\infty \sqrt{\delta},  \omega_{\rho_k,\sigma}(\sqrt{\delta})\}
\end{align*}
where $A_k=  k!\sum_{l=1}^k\frac{2^{l-1}}{(l-1)!(k-l)!} $.

Case II. Suppose $x'_i\notin [m_i,m_i+1]$ for some $i$.
Define $\vp=(p_1,\dots,p_k),\vp'=(p_1',\dots,p_k')\in \N_0^k$ by 
$$
\begin{array}{cc}
    p_i=m_i+1, \ p_i'=m_1' & \text{if } x_i<x'_i \\
    p_i=m_i, \ p_i'=m_1'+1 & \text{if } x_i>x'_i\\
    p_i=m_i=m'_i=p_i'& \text{if } x_i=x'_i
\end{array}
$$
for $i=1,\dots k$.
Then $p_i\in[m_i,m_i+1]$ and $p'_i\in[m'_i,m'_1+1]$ for all $i$. Hence by case I and because $\rho_k(\vp,\vp')\leq \rho_k(\vx,\vx')<\delta$, 
\begin{align*}
     |f(\vx)-  f(\vx')| &\leq |f(\vx)-  f(\vp)|+|\sigma(\vp)-\sigma(\vp')|+|f(\vp')-  f(\vx')|\\
     &\leq 2A_k\max \{2\|\sigma\|_\infty \sqrt{\delta},  \omega_{\rho_k,\sigma}(\sqrt{\delta})\} + \omega_{\rho_k,\sigma}(\delta).
\end{align*}
In both cases, $$|f(\vx)-  f(\vx')| \leq 2A_k\max \{2\|\sigma\|_\infty \sqrt{\delta},  \omega_{\rho_k,\sigma}(\sqrt{\delta})\} + \omega_{\rho_k,\sigma}(\delta).$$
Therefore $f\in \RO$. This completes the proof of Lemma \ref{lem:extend}.\\


\noindent \textbf{Data availability:} Data sharing not applicable to this article as no datasets were generated or analysed during the current study.


\end{document}